\documentclass[]{amsart}
\usepackage[T1]{fontenc}

\usepackage{amsmath,amssymb}
\usepackage{mathrsfs}
\usepackage{accents}
\usepackage[dvipdfmx]{graphicx}
\usepackage{float}

\usepackage{amsthm}
\theoremstyle{definition}
\newtheorem{THM}{Theorem}[section]
\newtheorem{LEM}[THM]{Lemma}
\newtheorem{PROP}[THM]{Proposition}

\newtheorem{DEF}[THM]{Definition}
\newtheorem{RMK}[THM]{Remark}
\newtheorem{EX}[THM]{Example}
\newtheorem*{THM*}{Theorem}
\newtheorem*{LEM*}{Lemma}
\newtheorem*{PROP*}{Proposition}
\newtheorem*{COR*}{Corollary}
\newtheorem*{DEF*}{Definition}
\newtheorem*{RMK*}{Remark}
\newtheorem*{EX*}{Example}

\makeatletter
  
  \@addtoreset{equation}{section}
\makeatother

\newcommand{\GLA}[1]{\mathbf{K}\hat{\pi}(#1)}
\newcommand{\UGLA}[1]{(\mathbf{K}G)\bar{\pi}(#1)}

\newcommand{\GSA}[1]{\operatorname{S}(#1)}
\newcommand{\UGSA}[1]{\bar{\operatorname{S}}(#1)}
\newcommand{\GPA}[1]{\operatorname{PS}(#1)}
\newcommand{\GCPA}[1]{\operatorname{CS}(#1)}

\newcommand{\UGPA}[1]{\overline{\operatorname{PS}}(#1)}
\newcommand{\SKQ}[1]{\operatorname{Sk}_q(#1)}
\newcommand{\SKH}[1]{\mathscr{K}_h(#1)}
\newcommand{\SK}[1]{\operatorname{Sk}(#1)}

\title{Poisson algebras of curves on bordered surfaces and skein quantization}
\author{Wataru Yuasa}
\date{}
\address{Department of Mathematics, Tokyo Institute of Technology, 2-12-1 Ookayama, Meguro-ku, Tokyo 152-8551, Japan}

\begin{document}
\begin{abstract}
We define a (co-)Poisson (co)algebra of curves on a bordered surface. 
A bordered surface is a surface whose boundary have marked points. 
Curves on the bordered surface are oriented loops and oriented arcs whose endpoints in the set of marked points. 
We define a (co-)Poisson (co)bracket on the symmetric algebra of a quotient of the vector space spanned by the regular homotopy classes of curves on the bordered surface by generalizing the Goldman bracket and the Turaev cobracket. 
Moreover, we define a Poisson algebra of unoriented curves on a bordered surface and show that a quantization of the Poisson algebra coincides with the skein algebra of the bordered surface defined by Muller.
\end{abstract}
\maketitle
\section{Introduction}
In this paper, we mainly discuss the Goldman Lie algebra of an oriented surface with marked points in its boundary. 
We review the Goldman Lie algebra and some related works.
The Goldman Lie algebra of a closed oriented surface $S$ was defined by Goldman~\cite{Goldman}.
He defined a Lie bracket, called the Goldman bracket, on the free $\mathbb{Z}$-module $\mathbb{Z}\hat{\pi}(S)$ spanned by the set of homotopy classes of oriented loops on $S$ where $\mathbb{Z}$ is the ring of integers.
The Goldman bracket is defined topologically by using the local intersection number and smoothing of two loops at intersection points.
He also defined the Goldman Lie algebra $\mathbb{Z}\bar{\pi}(S)$ of unoriented loops and embedded it into $\mathbb{Z}\hat{\pi}(S)$ as a Lie algebra.
The Lie algebra $\mathbb{Z}\bar{\pi}(S)$ was implicitly considered by Wolpert~\cite{Wolpert}.

On the other hand, 
the space $\operatorname{Hom}(\pi_1(S),G)/G$ of the conjugacy classes of representations of the fundamental group of $S$ to some Lie group $G$ has a symplectic structure introduced by Goldman~\cite{Goldman2}.
This is a generalization of Weil-Petersson form on the Teichm\"uller space of $S$.
Then $C^{\infty}(\operatorname{Hom}(\pi_1(S),G)/G)$ has a Poisson bracket.
Goldman constructed Lie algebra homomorphisms 
$\mathbb{Z}\hat{\pi}(S)\to C^{\infty}(\operatorname{Hom}(\pi_1(S),G)/G)$ and $\mathbb{Z}\bar{\pi}(S)\to C^{\infty}(\operatorname{Hom}(\pi_1(S),G)/G)$ 
for some Lie group $G$.
The definition of a Poisson algebra is described in Appendix.

Roger and Yang~\cite{Roger-Yang} introduced a variation of the Goldman Lie algebra of unoriented loops.
They defined a Poisson algebra $\mathcal{C}(S,P)$ of unoriented curves on an oriented surface $S$ with punctures $P$ in $\operatorname{int}S$ by generalizing the Goldman bracket to unoriented arcs with endpoints in $P$.
Following the approach of Bullock, Frohman and Kania-Bartoszy\'nska~\cite{BFK}, 
they constructed a Poisson homomorphism 
$\mathcal{C}(S,P)\to C^{\infty}(\mathcal{T}^{\text d}(S))$ by using the Weil-Petersson structure of the decorated Teichm\"uller space $\mathcal{T}^{\text d}(S)$ described by Mondello~\cite{Mondello}.

Another aspect of the Goldman Lie algebra appears as a quantization.
Let $\mathbb{C}$ be the field of complex numbers and $\mathbb{C}[h]$ the complex polynomial ring in one variable $h$. 
For a complex vector space $V$, 
we denote by $V[h]$ the set of all polynomials in one variable $h$ with coefficient in $V$.
\begin{DEF*}
A \emph{quantization} of a Poisson $\mathbb{C}$-algebra $B$ with the Poisson bracket $\{\cdot,\cdot\}_B$ is an associative $\mathbb{C}[h]$-algebra $A$ such that
\begin{enumerate}
\item $A$ is isomorphic to $B[h]$ as $\mathbb{C}[h]$-modules ($A$ is a free $\mathbb{C}[h]$-module),
\item $A/hA$ is isomorphic to $B$ as a $\mathbb{C}$-algebras and the isomorphism induces the Poisson algebra isomorphism.
\end{enumerate}
\end{DEF*}
If we denote the quotient map $p\colon A\to A/hA$, the above Poisson structure on $A/hA$ is defined by 
\[
 \{p(a),p(b)\}_{A/hA}=p\left(\frac{ab-ba}{h}\right)
\]
for any $a$ and $b$ in $A$.
A quantization of $B$ over the ring of complex formal power series $\mathbb{C}[[h]]$ is similarly defined by making $A$ a topological algebra and replacing the first condition with topological freeness of $A$ as $\mathbb{C}[[h]]$-module.
(See Drinfeld~\cite{Drinfeld} for details about the quantization.)

Turaev~\cite{Turaev1, Turaev2} gave quantizations of the Poisson algebras associated with the Goldman Lie algebras of oriented and unoriented loops on an oriented surface. 
In his papers, the freeness is not required in the definition of a quantization of a Poisson algebra.
The quantization was given by using a skein algebra of the surface.
The skein algebra is the module generated by the isotopy classes of links in the thickened surface with the skein relation. 
The multiplication of the skein algebra is given by superposition of links.
Hoste-Przytycki~\cite{Hoste-Przytycki} also considered the quantization of the Goldman Lie algebra.
Turaev defined a cobracket, called the Turaev cobracket, on the $\mathbb{Z}$-module generated by the homotopy classes of nontrivial loops on a surface in \cite{Turaev1, Turaev2}.
He gave a co-quantization of the co-Poisson coalgebra.
Moreover, he showed that the Goldman bracket and the Turaev cobracket satisfied compatibility conditions of a Lie bialgebra and gave a bi-quantization of the bi-Poisson bialgebra. 

Bullock, Frohman and Kania-Bartoszy\'nska~\cite{BFK} gave a quantization of the ring of $\operatorname{SL}_2(\mathbb{C})$-charactors of a compact oriented sruface by using the Kauffman bracket skein module of the surface.
Roger and Yang~\cite{Roger-Yang} defined the Poisson algebra $\mathcal{C}(S,P)$ of unoriented curves on punctured surface and gave a quantization by generalizing the Kauffman bracket skein algebra.

In this paper, 
we will define a Poisson algebra $\GPA{\Sigma}^{+}$ and a co-Poisson coalgebra $\GCPA{\Sigma}^{+}$ of the regular homotopy classes of oriented curves on a bordered surface $\Sigma$.
A bordered surface $\Sigma=(S,M)$ is an oriented surface $S$ with the set of marked points $M$ in the boundary of $S$.
The oriented curves on $\Sigma$ is oriented loops and arcs with endpoints in $M$.
The Poisson bracket and the co-Poisson cobracket are generalizations of a regular homotopy version of the Goldman bracket and the Turaev cobracket to those for oriented loops and arcs with endpoints in $M$. 
The Goldman bracket and the Turaev cobracket for regular homotopy classes of loops on a surface was considered by Turaev~\cite{Turaev2} and Kawazumi~\cite{Kawazumi}.
The Poisson algebra is also a generalization of the swapping algebra defined by Labourie~\cite{Labourie2}.
We will also define a Poisson algebra $\UGPA{\Sigma}^{+}$ of unoriented curves on $\Sigma$ and construct a quantization of a quotient Poisson algebra $\SK{\Sigma}$ of $\UGPA{\Sigma}^{+}$.
This quantization is given by using the skein algebra $\SKQ{\Sigma}$ defined by Muller~\cite{Muller}.
We denote by $\SKH{\Sigma}$ the $\mathbb{C}[[h]]$-algebra obtained by substituting $q=\exp(h/2)$ into $\mathbb{C}[q^{\frac{1}{2}},q^{-\frac{1}{2}}]$-algebra $\SKQ{\Sigma}$
\begin{THM*}[\ref{quantization}]
The map $\tilde{p}\colon\SKH{\Sigma}\to\SK{\Sigma}$ gives a quantization of $\SK{\Sigma}$.
\end{THM*}
On the other hand, 
a relation between $\SKQ{\Sigma}$ and the quantum cluster algebra $\mathcal{A}_q(\Sigma)$ obtained from a triangulation of $\Sigma$ was shown by Muller~\cite{Muller}.
(See Fomin and Zelevinsky~\cite{Fomin-Zelevinsky} and Fomin, Shapiro and Thurston~\cite{Fomin-Shapiro-Thurston} for details about cluster algebras and Berenstein and Zelevinsky~\cite{Berenstein-Zelevinsky} about quantum cluster algebras.)
\begin{THM*}[Theorem 9.8 of Muller~\cite{Muller}]
If a bordered surface $\Sigma$ is triangulable and has at least two marked points in each boundary component, then
\[
 \mathcal{A}_q(\Sigma)=\operatorname{Sk}_q^o(\Sigma).
\]
\end{THM*}
The above $\operatorname{Sk}_q^o(\Sigma)$ is the localized skein algebra defined by Muller~\cite{Muller}.
The specialization $q=1$ of $\mathcal{A}_q(\Sigma)$ becomes a commutative cluster algebra and $\mathcal{A}_1(\Sigma)=\operatorname{Sk}_1^o(\Sigma)$, see Corollary~{11.5(1)} of Muller~\cite{Muller}.
These results show that the Goldman Lie algebra of $\Sigma$ is related to the cluster algebra through the quantization.

This paper is organized as follows. 
Section~\ref{sec;preliminaries} is preliminary definitions.
Section~\ref{sec;Poisson} and \ref{sec;coPoisson} give the definitions of the Poisson algebra and the co-Poisson coalgebra of the regular homotopy classes of oriented curves on a bordered surface.
We define the Poisson algebra of unoriented curves on a bordered surface and give its quantization in Section~\ref{sec;quantization}.
Section~\ref{sec;Appendix} is a review of definitions of Poisson algebras, co-Poisson coalgebras and bi-Poisson bialgebras.
\paragraph{Acknowledgment}
The author would like to express his gratitude to his adviser, Hisaaki Endo, for his encouragement and helpful discussion and to Nariya Kawazumi for helpful comments. 
The author was supported by JSPS KAKENHI Grant Number 12J01252.

\section{Preliminaries}\label{sec;preliminaries}
This section gives definitions of main targets and notations which we treat in this paper. 

We consider an oriented, compact, connected and smooth surface $S$ with boundary. 
The boundary $\partial S$ contains a finite set $M$ of marked points. 
For each marked point in $M$, we fix an inward normal vector to $\partial S$.
We call the pair $\Sigma=(S,M)$ a \emph{bordered surface}. 
We allow a bordered surface which has boundary components with no marked points. 
We define curves on a bordered surface $\Sigma=(S,M)$.
\begin{DEF}
\ 
\begin{itemize}
\item A \emph{loop} on $\Sigma$ is an oriented immersed loop on $\operatorname{int}S=S\setminus\partial S$ whose self-intersection points consists only of transverse double points. 
The set of loops on $\Sigma$ is denoted by ${\sf Loops}(\Sigma)$. 
\item An \emph{arc} on $\Sigma$ is an oriented immersed arc on $S$ with endpoints in $M$ whose interior is disjoint from $\partial S$. 
Moreover, 
we require that self-intersection points of the arc in $\operatorname{int}S$ are transverse double points and the arc is tangent to inward normal vector on $M$ at the endpoints.
The set of arcs on $\Sigma$ is denoted by ${\sf Arcs}(\Sigma)$. 
\item A \emph{curve} on $\Sigma$ is a loop or an arc on $\Sigma$. 
We denote the set of curves by ${\sf Curves}(\Sigma)={\sf Loops}(\Sigma)\cup{\sf Arcs}(\Sigma)$. 
\item A subset $X$ of ${\sf Curves}(\Sigma)$ is \emph{generic} or \emph{generic curves} if $X$ has only finite transverse double points in $\operatorname{int}S$. 
\end{itemize}
\end{DEF}
\begin{RMK}
Dylan Thurston call generic subsets of ${\sf Curves}(\Sigma)$ curve diagrams on $\Sigma$ and elements of ${\sf Curves}(\Sigma)$ connected curve diagrams in his paper~\cite{ThurstonD}. 
\end{RMK}
\begin{figure}[h]
\centering
\includegraphics[scale=0.5, clip]{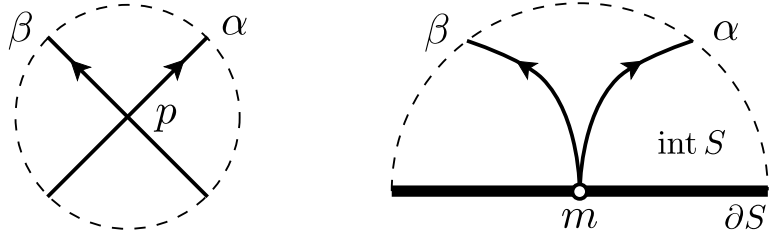}
\caption{intersections of curves}
\label{intersection}
\end{figure}
We consider the following equivalence relation on ${\sf Curves}(\Sigma)$. 
For generic subset $\{\alpha,\beta\}$ of ${\sf Curves}(\Sigma)$, 
the curve $\alpha$ is \emph{regularly homotopic} to $\beta$ 
if there is a finite sequence of generic curves $\{f_1,f_2,\dots,f_n\}\subset{\sf Curves}(\Sigma)$ 
such that $f_{i+1}$ is obtained from $f_{i}$ by ambient isotopy of $S$ fixing a collar of $\partial S$ 
or one of the moves ($\Omega$2-1), ($\Omega$2-2), ($\Omega$3) and ($\Omega$b2-1)--($\Omega$b2-4) illustrated in Figure~\ref{homotopymoves}. 
\begin{figure}[h]
\centering
\includegraphics[scale=0.5, clip]{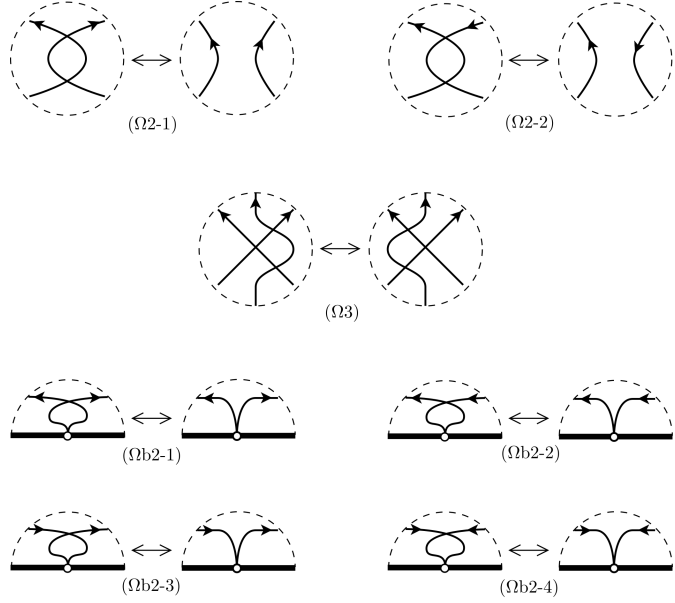}
\caption{local moves of curves}
\label{homotopymoves}
\end{figure}
We denote the set of regular homotopy classes of curves on $\Sigma$ by $\hat{\pi}(\Sigma)^{+}$.

We define the local intersection number of two curves on $\Sigma$. 
Let $\{\alpha,\beta\}\subset{\sf Curves}(\Sigma)$ be generic. 
We define the local intersection number of $\alpha$ and $\beta$ at an intersection point $p$. 
\begin{DEF}[the local intersection number]
If $p$ is in $S\setminus\partial S$, 
\begin{equation*}
\varepsilon(p;\alpha,\beta) =
\begin{cases}
1  & \text{if}\ \vcenter{\hbox{\includegraphics[width=0.8cm]{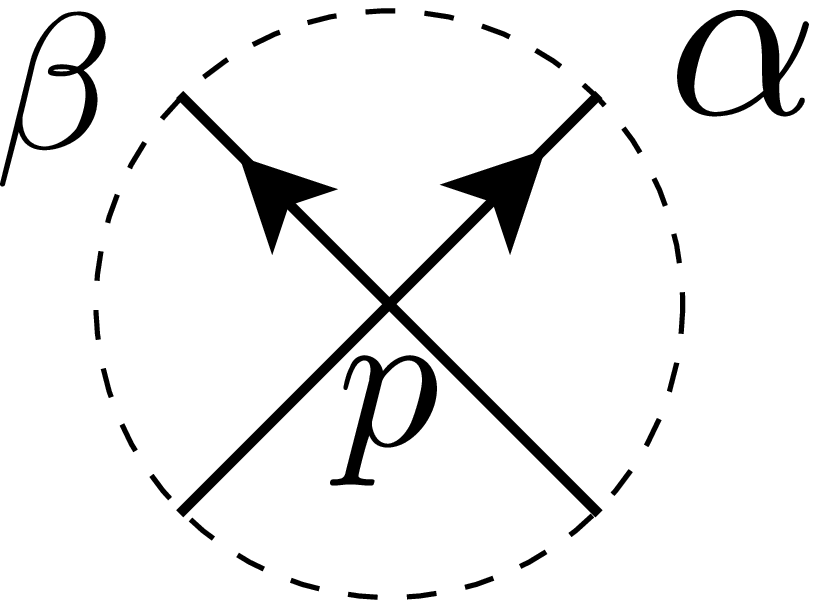}}},\\
-1 & \text{otherwise}.
\end{cases}
\end{equation*}
If $p$ is a marked point in $M$, 
\begin{equation*}
\varepsilon(p;\alpha,\beta) =
\begin{cases}
\frac{1}{2}  & \text{if} \ \vcenter{\hbox{\includegraphics[width=1cm]{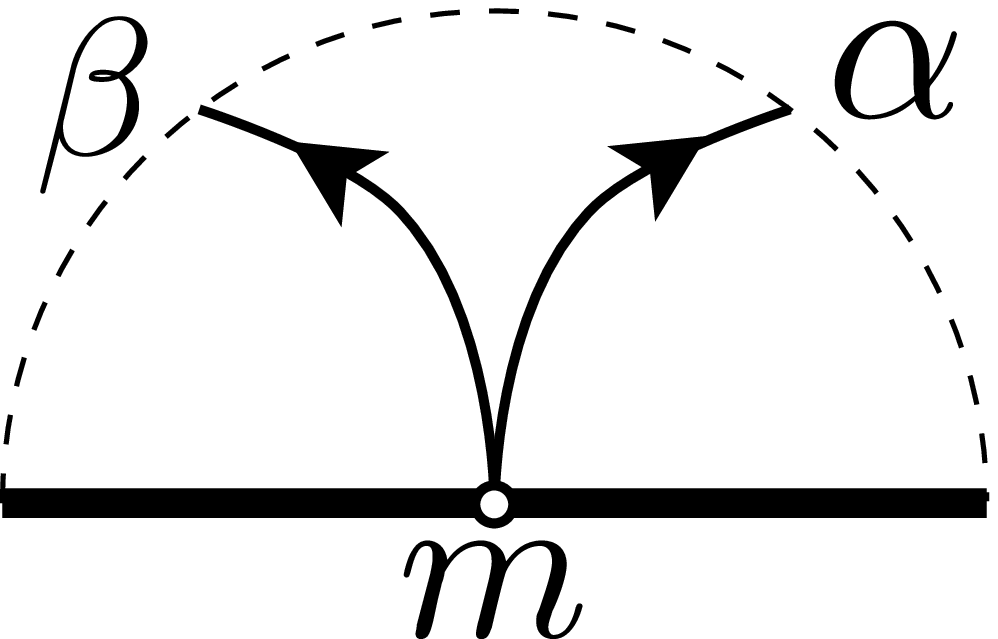}}}, 
\ \vcenter{\hbox{\includegraphics[width=1cm]{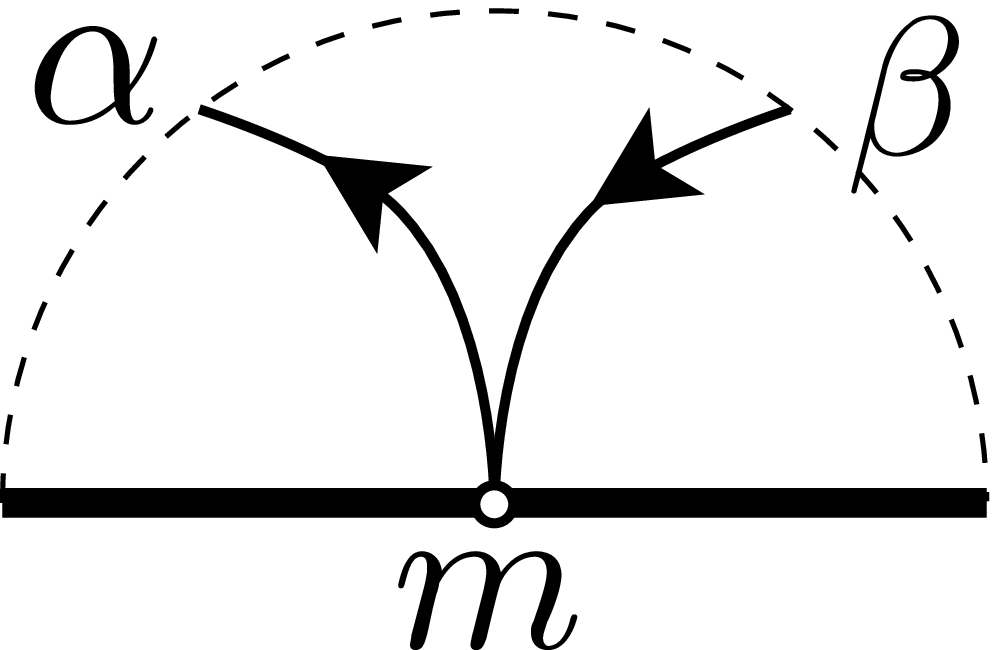}}}, 
\ \vcenter{\hbox{\includegraphics[width=1cm]{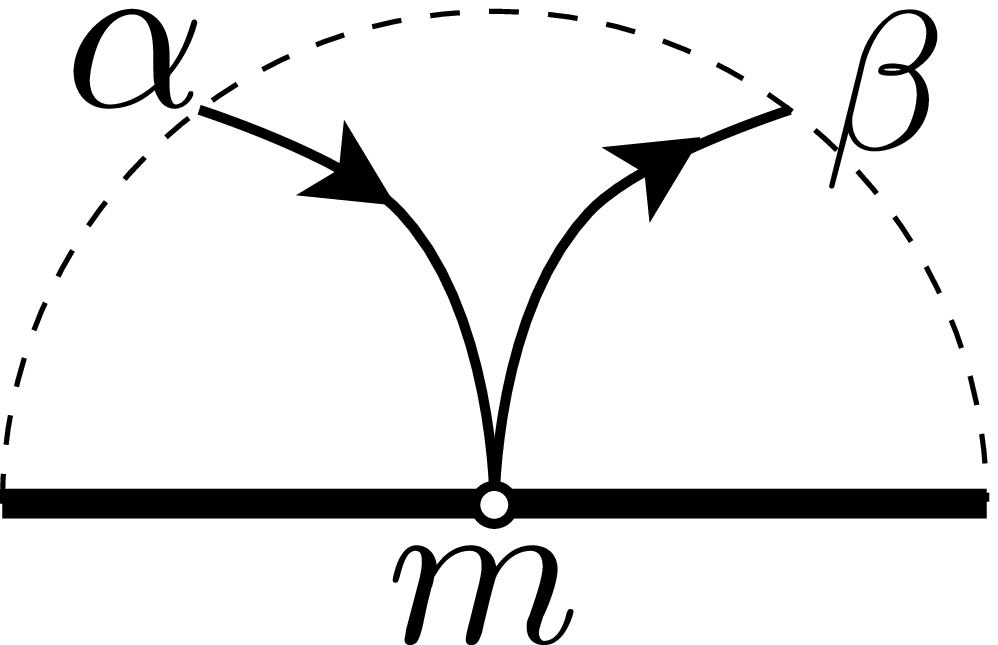}}}, 
\ \vcenter{\hbox{\includegraphics[width=1cm]{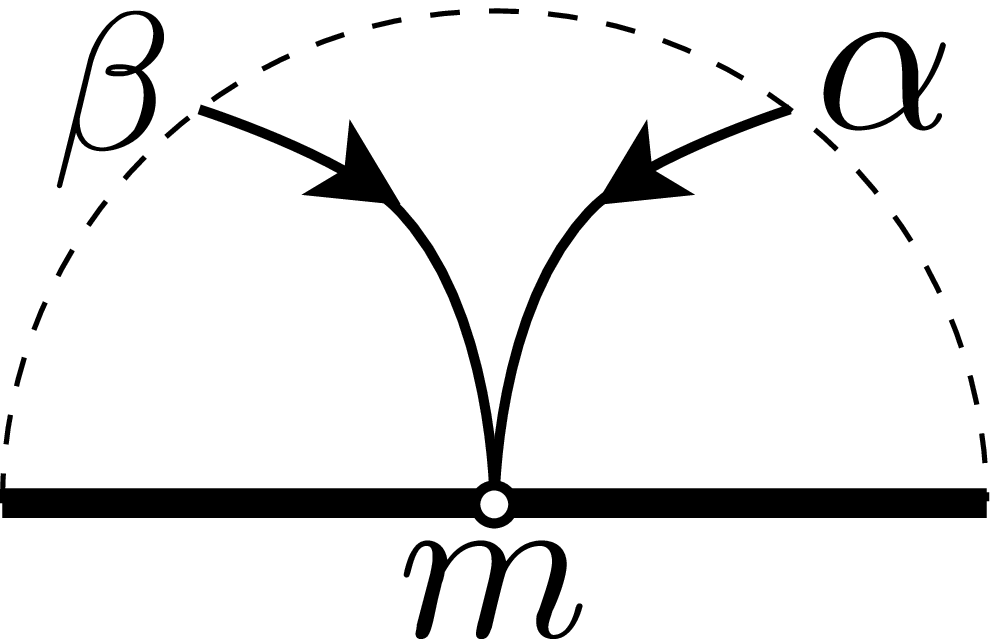}}},\\
-\frac{1}{2} & \text{otherwise.}
\end{cases}
\end{equation*}
\end{DEF}
We next introduce notations which represent subarcs or based subloops of curves on $\Sigma$. 
We take a curve $\alpha$ in ${\sf Curves}(\Sigma)$ and distinct points $p$ and $q$ on $\alpha$. 
\begin{itemize}
\item $\alpha_p^q\colon I=[0,1]\to S$ is obtained along $\alpha$ such that $\alpha_p^q(0)=p$, $\alpha_p^q(1)=q$ and $p,q\notin\alpha_p^q(I\setminus\partial I)$.
\end{itemize}
If $\alpha\in{\sf Loops}(\Sigma)$ and $p$ is not a self-intersection point of $\alpha$, 
then
\begin{itemize}
\item $\alpha_p$ is the based loop at $p$ obtained from $\alpha$.
\end{itemize}
If $p$ is a self-intersection point of $\alpha$, 
then we take two tangent vectors $v_1$ and $v_2$ of $\alpha$ at $p$ such that $(v_1,v_2)$ coincides with orientation of $S$. 
For $i=1,2$, We can obtain the based loop $\alpha_p^{(i)}$ which goes from $p$ in the direction of $v_i$ along $\alpha$ and stops at the first meeting $p$. 
Let $\alpha$ be in ${\sf Arcs}(\Sigma)$ and $p$ a self-intersection point of $\alpha$ in $\operatorname{int}S$, then we can define either $\alpha_p^{(1)}$ or $\alpha_p^{(2)}$. 
We call the self-intersection point $p$ is type(I) (resp. type(II)) if we can define $\alpha_p^{(1)}$ (resp. $\alpha_p^{(2)}$). (See Figure \ref{subloop}) 
For $x,y$ in $M$, we can decompose an arc $\alpha=\alpha_x^y$ on $\Sigma$ to three parts $\alpha_x^p$, $\alpha_p^{(1)}$ (resp. $\alpha_p^{(2)}$) and $\alpha_p^x$ if $p$ is type(I) (resp. type(II)). 
\begin{figure}[h]
\centering
\includegraphics[scale=0.5, clip]{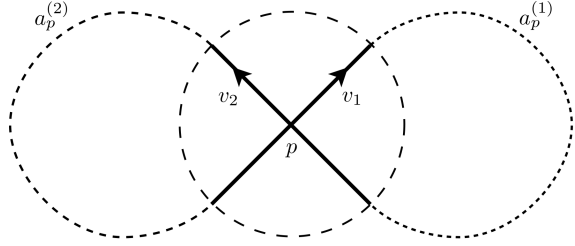}
\caption{subloops $a_p^{(1)}$ and $a_p^{(2)}$}
\label{subloop}
\end{figure}

\section{Poisson algebras of curves on bordered surfaces}\label{sec;Poisson}
We will define a Poisson bracket for two regular homotopy classes of curves on $\Sigma$ in this section. 
Let us denote by $\mathbf{K}$ a commutative associative ring with unit containing the field of rational numbers, $\langle t \rangle$ the infinite cyclic group generated by $t$ and $\mathbf{K}[t,t^{-1}]=\mathbf{K}\langle t\rangle$ the Laurent polynomial ring. 

Let $\GLA{\Sigma}^{+}$ be the free $\mathbf{K}$-module on $\hat{\pi}(\Sigma)^{+}$. 
The Laurent polynomial ring $\mathbf{K}[t,t^{-1}]$ acts on $\GLA{\Sigma}^{+}$ such that the action of $t$ (resp. $t^{-1}$) on a curve on $\Sigma$ is an insertion of a positively (resp. negatively) oriented monogon into interior of the curve. 
Thus $\GLA{\Sigma}^{+}$ is a $\mathbf{K}[t,t^{-1}]$-module. 
Then we define the $\mathbf{K}[t,t^{-1}]$-module $Z(\Sigma)^{+}$ 
as the quotient of $\GLA{\Sigma}^{+}$ by the $\mathbf{K}[t,t^{-1}]$-submodule $C_M^{+}$ generated by monogons with vertices in $M$. 
The monogons are contractible arcs on $\Sigma$ which have no self-intersection points.
We denote by $\left|\alpha\right|$ the element of $Z(\Sigma)$ represented by $\alpha$ in $\mathsf{Curves}(\Sigma)$. 
Let us denote by $\GSA{\Sigma}^{+}$ the symmetric $\mathbf{K}[t,t^{-1}]$-algebra $\operatorname{Sym}(Z(\Sigma)^{+})$.
\begin{RMK}\ 
\begin{enumerate}
\item Let $\hat{\pi}(\Sigma)$ be the set of homotopy classes of curves on $\Sigma$ and $C_M$ the $\mathbf{K}$-submodule of $\GLA{\Sigma}$ generated by contractible arcs on $\Sigma$. 
Then $Z(\Sigma)^{+}/(t-1)Z(\Sigma)^{+}=Z(\Sigma)=\GLA{\Sigma}/C_M$. 
We denote the symmetric $\mathbf{K}$-algebra of $Z(\Sigma)$ by $\GSA{\Sigma}$. 
\item A curve diagram is an generic immersion of a $1$-manifold $(X,\partial X)\to(S,M)$ which defined by Thurston~\cite{ThurstonD}. 
In our own words, the set of curve diagrams is the set of all generic curves on $\Sigma$.
He define the commutative associative monoid $\mathbf{C}(\Sigma)$ of regular homotopy classes of curve diagrams on $\Sigma$ whose unit is the empty diagram. 
The multiplication of two curve diagrams is given by the union of two diagrams. 
A subset $X$ of $\mathsf{Curves}(\Sigma)$ defines an element of $\operatorname{Sym}(\GLA{\Sigma}^{+})$ by multiplication of regular homotopy classes of elements in $X$.
The monoid ring $\mathbf{K}\mathbf{C}(\Sigma)$ is isomorphic to $\operatorname{Sym}(\GLA{\Sigma}^{+})$ through the above map. 
\end{enumerate}
\end{RMK}

We next define a bracket on $\GSA{\Sigma}^{+}$. 
We take curves $a$ and $b$ in ${\sf Loops}(\Sigma)$ and $\gamma=\gamma_x^y$ and $\eta=\eta_z^w$ in ${\sf Arcs}(\Sigma)$ such that $\{a,b,\gamma,\eta\}$ is generic and $x,y,z$ and $w$ are in $M$. 
We define a bracket $\{\cdot,\cdot\}'\colon{\sf Curves}(\Sigma)\times{\sf Curves}(\Sigma)\to \GSA{\Sigma}^{+}$ as the following.
\begin{align}
\{a,b\}'=&\!\sum_{p\in a\cap b}\varepsilon(p;a,b)\left|a_pb_p\right|,\label{ll}\\
\{a,\eta\}'=&\!\sum_{p\in a\cap \eta}\varepsilon(p;a,\eta)\left|\eta_z^p a_p\eta_p^w\right|,\label{la}\\
\{\gamma,b\}'=&\!\sum_{p\in \gamma\cap b}\varepsilon(p;\gamma,b)\left|\gamma_x^p b_p \gamma_p^y\right|,\label{al}\\
\{\gamma,\eta\}'=&\!\sum_{p\in \gamma\cap\eta}\varepsilon(p;\gamma,\eta)\left|\gamma_x^p\eta_p^w\right|\left|\eta_z^p\gamma_p^y\right|\label{aa}\\
=&\!\sum_{p\in \gamma\cap\eta\cap\operatorname{int}S}\varepsilon(p;\gamma,\eta)\left|\gamma_x^p\eta_p^w\right|\left|\eta_z^p\gamma_p^y\right|\nonumber\\
&+\delta_{x,z}\varepsilon(x;\gamma,\eta)\left|\gamma\right|\left|\eta\right|
+\delta_{y,w}\varepsilon(x;\gamma,\eta)\left|\eta\right|\left|\gamma\right|,\nonumber
\end{align}
where $\delta$ is the Kronecker delta function.
We know that (\ref{ll})--(\ref{aa}) are invariant under local moves ($\Omega$2-1) and ($\Omega$3). (See Goldman~\cite{Goldman}.) 
We can also confirm that (\ref{ll})--(\ref{aa}) are invariant under ($\Omega$2-2) 
and (\ref{aa}) is invariant under ($\Omega$b2-1)--($\Omega$b2-4) by comparing terms of the bracket coming from intersection points illustrated in diagrams of local moves.
For example, 
\begin{align*}
\text{the terms coming from}\ \{\alpha,\beta\}'\ \text{of} 
\ \vcenter{\hbox{\includegraphics[width=1cm]{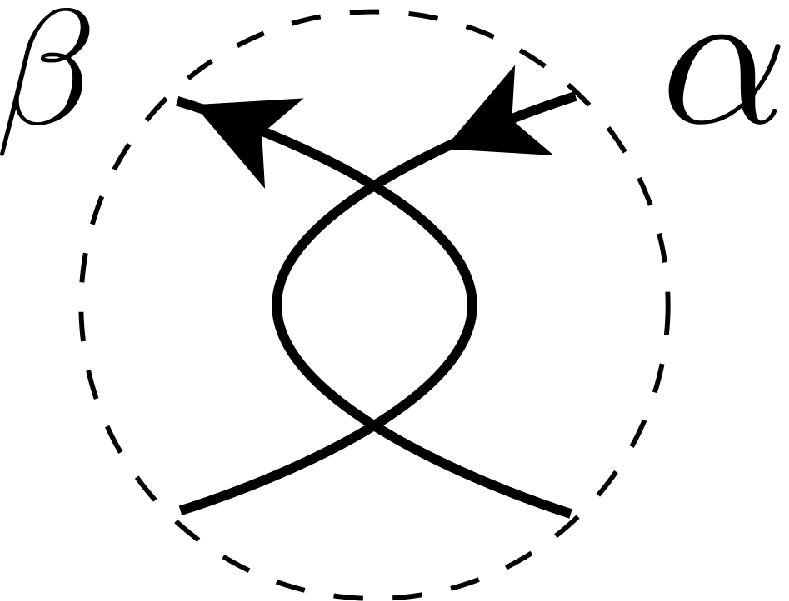}}}
=&\ \vcenter{\hbox{\includegraphics[width=1cm]{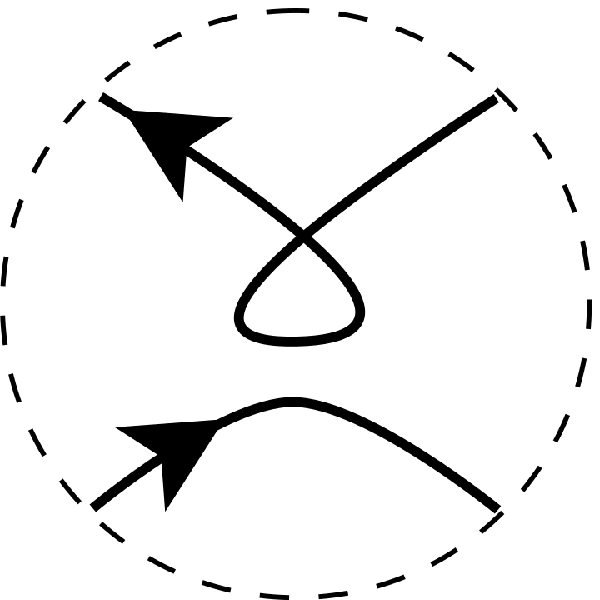}}}
-\ \vcenter{\hbox{\includegraphics[width=1cm]{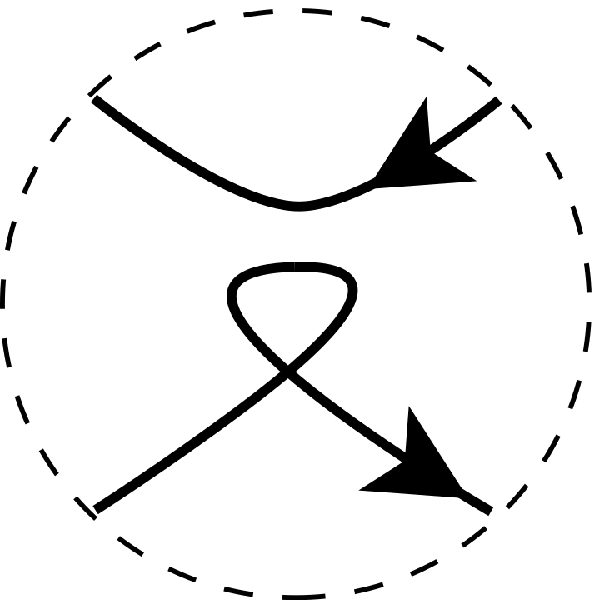}}}\\
=&t\ \vcenter{\hbox{\includegraphics[width=1cm]{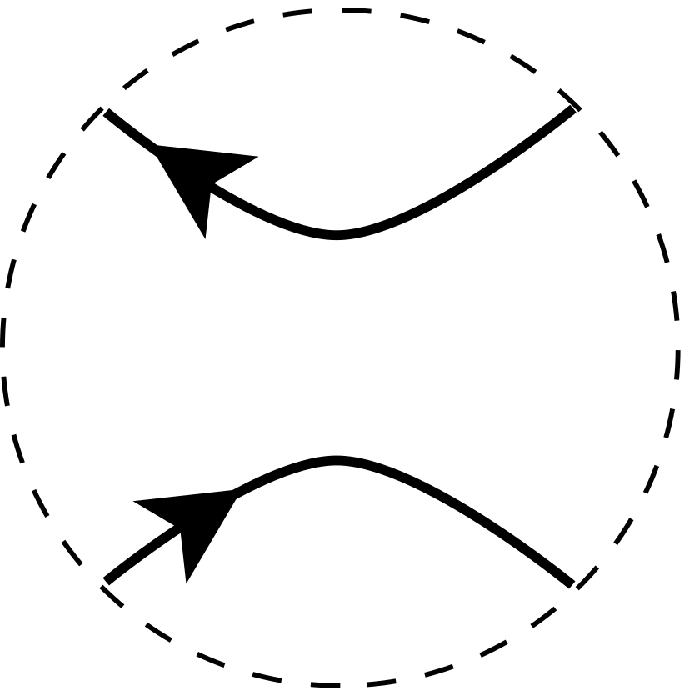}}}
-t\ \vcenter{\hbox{\includegraphics[width=1cm]{invariance4}}}\\
=&0,\\
\text{the terms coming from}\ \{\alpha,\beta\}'\ \text{of} 
\ \vcenter{\hbox{\includegraphics[width=1cm]{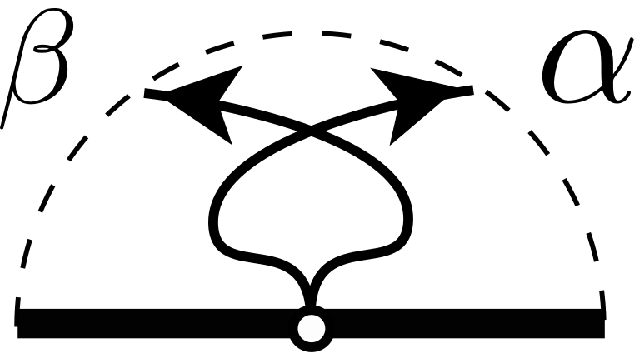}}}
=&\ \vcenter{\hbox{\includegraphics[width=1cm]{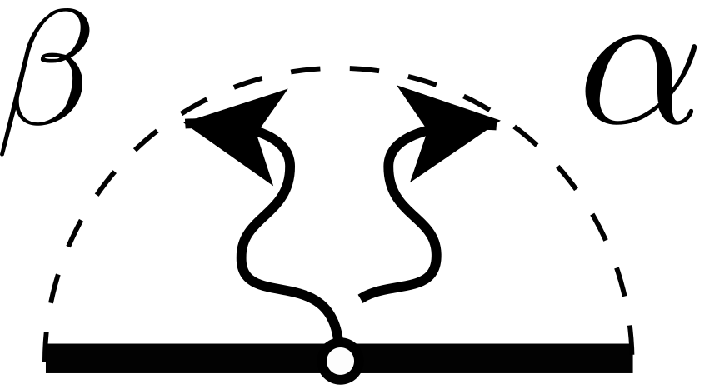}}}
-\frac{1}{2}\ \vcenter{\hbox{\includegraphics[width=1cm]{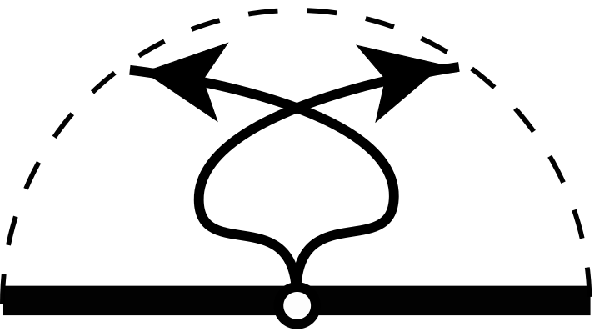}}}\\
=&\frac{1}{2}\ \vcenter{\hbox{\includegraphics[width=1cm]{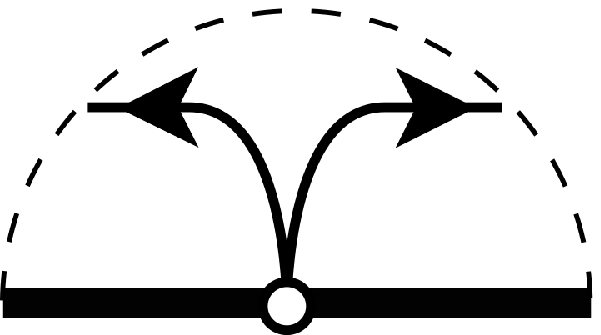}}}\\
=&\ \text{the terms coming from}\ \{\alpha,\beta\}'\ \text{of} 
\ \vcenter{\hbox{\includegraphics[width=1cm]{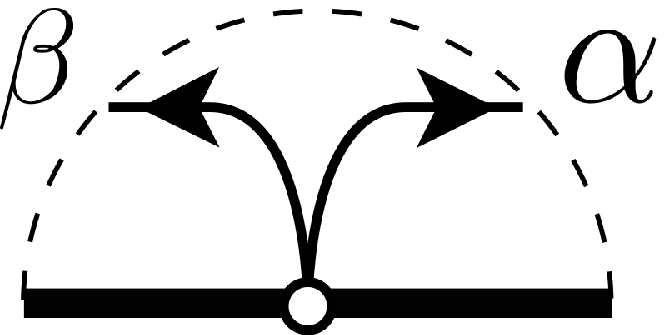}}}.
\end{align*}
When $\gamma$ or $\eta$ is a monogon with vertex in $M$, then the bracket vanishes.
Therefore we can define the $\mathbf{K}[t,t^{-1}]$-bilinear map 
$\nabla'=\{\cdot,\cdot\}'\colon Z(\Sigma)^{+}\otimes_{t} Z(\Sigma)^{+}\to \GSA{\Sigma}^{+}$. 
The symbol $\otimes_t$ denotes the tensor product over $\mathbf{K}[t,t^{-1}]$.
Using the Leibniz rule, 
the bracket is extended to $\nabla=\{\cdot,\cdot\}\colon \GSA{\Sigma}^{+}\otimes_t \GSA{\Sigma}^{+}\to \GSA{\Sigma}^{+}$. 

We can easily prove the following lemma.
\begin{LEM}\label{nabla}\ 
\begin{enumerate}
\item If $\nabla'$ is skew symmetric, then $\nabla$ is skew symmetric.
\item If $\nabla'$ is skew symmetric and satisfies the Jacobi identity, then $\nabla$ satisfies the Jacobi identity.
\end{enumerate}
\begin{proof}
Use the Leibniz rule and an induction on the degree of symmetric power.
\end{proof}
\end{LEM}
\begin{PROP}\label{Jacobi}
$\nabla'$ is skew symmetric and satisfies the Jacobi identity.
\end{PROP}
\begin{proof}
We can see $\nabla'$ is skew symmetric by the definition of the bracket and the local intersection number. 
We prove the Jacobi identity. 
We take curves $a,b$ and $c$ in ${\sf Loops}(\Sigma)$ and $\gamma=\gamma_x^y,\eta=\eta_z^w$ and $\zeta=\zeta_u^v$ in ${\sf Arcs}(\Sigma)$ where $x,y,z,w,u$ and $v$ are in $M$. 
We assume that $\{a,b,c,\gamma,\eta,\zeta\}$ is generic. 
For $a,b$ and $c$, the Jacobi identity is satisfied because the definition of $\nabla'$ for loops coincides with the Goldman bracket. We have to prove three equalities: 
\begin{align}
\{\{a,b\}',\gamma\}'+\{\{b,\gamma\}',a\}'+\{\{\gamma,a\}',b\}'=&0,\label{lla}\\
\{\{a,\gamma\}',\eta\}'+\{\{\gamma,\eta\}',a\}'+\{\{\eta,a\}',\gamma\}'=&0,\label{laa}\\
\{\{\gamma,\eta\}',\zeta\}'+\{\{\eta,\zeta\}',\gamma\}'+\{\{\zeta,\gamma\}',\eta\}'=&0.\label{aaa}
\end{align}
We can prove these equalities by direct calculation. 
We leave the readers the calculation of (\ref{lla}) because it is simpler than (\ref{laa}) and (\ref{aaa}). 
We will calculate (\ref{laa}) and (\ref{aaa}) according to the definition of the bracket using the Kronecker delta function. 
First, we calculate (\ref{laa}).
\begin{align}
\{\{a,\gamma\}',\eta\}'\nonumber\\
=&\!\sum_{p\in a\cap\gamma}\sum_{q\in\gamma_x^p\cap\eta\cap\operatorname{int}S}
\varepsilon(p;a,\gamma)\varepsilon(q;\gamma,\eta)
\left|\gamma_x^q\eta_q^w\right|\left|\eta_z^q\gamma_q^pa_p\gamma_p^y\right|\label{laa1}\\
&+\!\sum_{p\in a\cap\gamma}\sum_{q\in a\cap\eta}
\varepsilon(p;a,\gamma)\varepsilon(q;a,\eta)
\left|\gamma_x^q a_p^q\eta_q^w\right|\left|\eta_z^q a_q^p\gamma_p^y\right|\label{laa2}\\
&+\!\sum_{p\in a\cap\gamma}\sum_{q\in\gamma_p^y\cap\eta\cap\operatorname{int}S}
\varepsilon(p;a,\gamma)\varepsilon(q;\gamma,\eta)
\left|\gamma_x^p a_p\gamma_p^q\eta_q^w\right|\left|\eta_z^q\gamma_q^y\right|\label{laa3}\\
&+(\text{ the terms of degree $1$ of the Kronecker delta function})\tag{$\delta$1},
\end{align}
\begin{align}
\{\{\gamma,\eta\}',a\}'\nonumber\\
=&\sum_{p\in\gamma\cap\eta\cap\operatorname{int}S}\sum_{q\in\gamma_x^p\cap a}
\varepsilon(p;\gamma,\eta)\varepsilon(q;\gamma,a)
\left|\gamma_x^q a_q\gamma_q^p\eta_p^w\right|\left|\eta_z^p\gamma_p^y\right|\label{aal1}\\
&+\sum_{p\in\gamma\cap\eta\cap\operatorname{int}S}\sum_{q\in\eta_p^w\cap a}
\varepsilon(p;\gamma,\eta)\varepsilon(q;\eta,a)
\left|\gamma_x^p\eta_p^q a_q\eta_q^w\right|\left|\eta_z^p\gamma_p^y\right|\label{aal2}\\
&+\sum_{p\in\gamma\cap\eta\cap\operatorname{int}S}\sum_{q\in\eta_z^p\cap a}
\varepsilon(p;\gamma,\eta)\varepsilon(q;\eta,a)
\left|\gamma_x^p\eta_p^w\right|\left|\eta_z^q a_q\eta_q^p\gamma_p^y\right|\label{aal3}\\
&+\sum_{p\in\gamma\cap\eta}\sum_{q\in\gamma_p^y\cap a}
\varepsilon(p;\gamma,\eta)\varepsilon(q;\gamma,a)
\left|\gamma_x^p\eta_p^w\right|\left|\eta_z^p\gamma_p^q a_q\gamma_q^y\right|\label{aal4}\\
&+(\text{ the terms of degree $1$ of the Kronecker delta function})\tag{$\delta$2},
\end{align}
\begin{align}
\{\{\eta,a\}',\gamma\}'\nonumber\\
=&\sum_{p\in\eta\cap a}\sum_{q\in\eta_z^p\cap\gamma\cap\operatorname{int}S}
\varepsilon(p;\eta,a)\varepsilon(q;\eta,\gamma)
\left|\eta_z^q\gamma_q^y\right|\left|\gamma_x^q\eta_q^p a_p\eta_p^w\right|\label{ala1}\\
&+\sum_{p\in\eta\cap a}\sum_{q\in a\cap\gamma}
\varepsilon(p;\eta,a)\varepsilon(q;a,\gamma)
\left|\eta_z^p a_p^q\gamma_q^y\right|\left|\gamma_x^q a_q^p\eta_p^w\right|\label{ala2}\\
&+\sum_{p\in\eta\cap a}\sum_{q\in a\cap\gamma}
\varepsilon(p;\eta,a)\varepsilon(q;a,\gamma)
\left|\eta_z^p a_p^q\gamma_q^y\right|\left|\gamma_x^q a_q^p\eta_p^w\right|\label{ala3}\\
&+(\text{ the terms of degree $1$ of the Kronecker delta function}).\tag{$\delta$3}
\end{align}
The sum of (\ref{laa1}) and (\ref{laa3}) is the sum of subset of curves obtained from $a\cup\gamma\cup\eta$ by smoothing at any pair of intersection points $p\in a\cap\gamma$ and $q\in\gamma\cap\eta$ with the sign $\varepsilon(p;a,\gamma)\varepsilon(q;\gamma,\eta)$. 
We remark that a subset $\{\alpha_1,\alpha_2,\dots,\alpha_n\}$ of curves on $\Sigma$ gives the element $\left|\alpha_1\right|\left|\alpha_2\right|\cdots\left|\alpha_n\right|$ of $\GSA{\Sigma}^+$.
The sum of (\ref{aal1}) and (\ref{aal4}) is also the sum of the curves with the opposite sign $\varepsilon(q;\gamma,\eta)\varepsilon(p;\gamma,a)$. 
Therefore, these sums cancel each other out. 
Similarly, the sum of (\ref{ala1}) and (\ref{ala3}) cancels out the sum of (\ref{aal2}) and (\ref{aal3}). 
Clearly, (\ref{laa2}) is opposite in sign to (\ref{ala2}). It is only necessary to confirm that the terms of degree $1$ of the Kronecker delta function vanishes. 
\begin{align*}
\text{($\delta$1)}
=&\sum_{p\in a\cap\gamma}
\delta_{y,w}\varepsilon(p;a,\gamma)\varepsilon(y;\gamma,\eta)
\left|\gamma_x^p a_p\gamma_p^y\right|\left|\eta\right|\\
&+\sum_{p\in a\cap\gamma}
\delta_{x,z}\varepsilon(p;a,\gamma)\varepsilon(x;\gamma,\eta)
\left|\eta\right|\left|\gamma_x^p a_p\gamma_p^y\right|,
\end{align*}
\begin{align*}
\text{($\delta$2)}
=&\sum_{p\in\gamma\cap a}
\delta_{y,w}\varepsilon(p;\gamma,a)\varepsilon(y;\gamma,\eta)
\left|\gamma_x^p a_p\gamma_p^y\right|\left|\eta\right|\\
&+\sum_{p\in\eta\cap a}
\delta_{y,w}\varepsilon(p;\eta,a)\varepsilon(y;\gamma,\eta)
\left|\gamma\right|\left|\eta_z^p a_p\eta_p^y\right|\\
&+\sum_{p\in\eta\cap a}
\delta_{x,z}\varepsilon(p;\eta,a)\varepsilon(x;\gamma,\eta)
\left|\eta_z^p a_p\eta_p^w\right|\left|\gamma\right|\\
&+\sum_{p\in\gamma\cap a}
\delta_{x,w}\varepsilon(p;\gamma,a)\varepsilon(x;\gamma,\eta)
\left|\eta\right|\left|\gamma_x^p a_p\gamma_p^y\right|,
\end{align*}
\begin{align*}
\text{($\delta$3)}
=&\sum_{p\in\eta\cap a}
\delta_{w,y}\varepsilon(p;\eta,a)\varepsilon(w;\eta,\gamma)
\left|\eta_z^p a_p\eta_p^w\right|\left|\gamma\right|\\
&+\sum_{p\in\eta\cap a}
\delta_{z,x}\varepsilon(p;\eta,a)\varepsilon(z;\eta,\gamma)
\left|\eta_z^p a_p\eta_p^w\right|\left|\gamma\right|.
\end{align*}
From the above, we see that the sum of ($\delta$1) and ($\delta$3) cancels out ($\delta$2). Consequently, the Jacobi identity (\ref{laa}) holds. Finally, we prove (\ref{aaa}). 
\begin{align}
\{\{\gamma,\eta\}'\zeta\}'
=&\sum_{p\in\gamma\cap\eta\cap\operatorname{int}S}\sum_{q\in\gamma_x^p\cap\zeta\cap\operatorname{int}S}
\varepsilon(p;\gamma,\eta)\varepsilon(q;\gamma,\zeta)
\left|\gamma_x^q\zeta_q^v\right|\left|\zeta_u^q\gamma_q^p\eta_p^w\right|\left|\eta_z^p\gamma_p^y\right|\label{aaa1}\\
&+\sum_{p\in\gamma\cap\eta\cap\operatorname{int}S}\sum_{q\in\eta_p^w\cap\zeta\cap\operatorname{int}S}
\varepsilon(p;\gamma,\eta)\varepsilon(q;\eta,\zeta)
\left|\gamma_x^p\eta_p^q\zeta_q^v\right|\left|\zeta_u^q\eta_q^w\right|\left|\eta_z^p\gamma_p^y\right|\label{aaa2}\\
&+\sum_{p\in\gamma\cap\eta\cap\operatorname{int}S}\sum_{q\in\eta_z^p\cap\zeta\cap\operatorname{int}S}
\varepsilon(p;\gamma,\eta)\varepsilon(q;\eta,\zeta)
\left|\gamma_x^p\eta_p^w\right|\left|\eta_z^q\zeta_q^v\right|\left|\zeta_u^q\eta_q^p\gamma_p^y\right|\label{aaa3}\\
&+\sum_{p\in\gamma\cap\eta\cap\operatorname{int}S}\sum_{q\in\gamma_p^y\cap\zeta\cap\operatorname{int}S}
\varepsilon(p;\gamma,\eta)\varepsilon(q;\gamma,\zeta)
\left|\gamma_x^p\eta_p^w\right|\left|\eta_z^p\gamma_p^q\zeta_q^v\right|\left|\zeta_u^q\gamma_q^y\right|\label{aaa4}\\
&+(\text{the terms of degree $1$ of the Kronecker delta function})\tag{$\delta$4}\\
&+(\text{the terms of degree $2$ of the Kronecker delta function})\tag{$\delta\delta$}.
\end{align}
The sum of (\ref{aaa1}) and (\ref{aaa4}) appears in a cyclic permutation of the symbols $(\gamma,\eta,\zeta)$ at the sum of (\ref{aaa2}) and (\ref{aaa3}) with the opposite sign. 
We can confirm that the terms of (\ref{aaa}) which has no Kronecker delta function vanishes in the same way. 
We can calculate remaining parts ($\delta$4) and ($\delta\delta$) as the following.
\begin{align}
\text{($\delta$4)}
=&\sum_{p\in\gamma\cap\zeta\cap\operatorname{int}S}
\delta_{x,z}\varepsilon(x;\gamma,\eta)\varepsilon(p;\gamma,\zeta)
\left|\gamma_x^p\zeta_p^v\right|\left|\zeta_u^p\gamma_p^y\right|\left|\eta\right|
\label{d1}\\
&+\sum_{p\in\eta\cap\zeta\cap\operatorname{int}S}
\delta_{x,z}\varepsilon(x;\gamma,\eta)\varepsilon(p;\eta,\zeta)
\left|\gamma\right|\left|\eta_z^p\zeta_p^v\right|\left|\zeta_u^p\eta_p^w\right|
\label{d2}\\
&+\sum_{p\in\gamma\cap\zeta\cap\operatorname{int}S}
\delta_{y,w}\varepsilon(y;\gamma,\eta)\varepsilon(p;\gamma,\zeta)
\left|\eta\right|\left|\gamma_x^p\zeta_p^v\right|\left|\zeta_u^p\gamma_p^y\right|
\label{d3}\\
&+\sum_{p\in\eta\cap\zeta\cap\operatorname{int}S}
\delta_{y,w}\varepsilon(y;\gamma,\eta)\varepsilon(p;\eta,\zeta)
\left|\eta_z^p\zeta_p^v\right|\left|\zeta_u^p\eta_p^w\right|\left|\gamma\right|
\label{d4}\\
&+\sum_{p\in\gamma\cap\eta\cap\operatorname{int}S}
\delta_{x,u}\varepsilon(x;\gamma,\zeta)\varepsilon(p;\gamma,\eta)
\left|\gamma_x^p\eta_p^w\right|\left|\zeta\right|\left|\eta_z^p\gamma_p^y\right|
\label{d5}\\
&+\sum_{p\in\gamma\cap\eta\cap\operatorname{int}S}
\delta_{w,v}\varepsilon(w;\eta,\zeta)\varepsilon(p;\gamma,\eta)
\left|\zeta\right|\left|\gamma_x^p\eta_p^w\right|\left|\eta_z^p\gamma_p^y\right|
\label{d6}\\
&+\sum_{p\in\gamma\cap\eta\cap\operatorname{int}S}
\delta_{z,u}\varepsilon(z;\eta,\zeta)\varepsilon(p;\gamma,\eta)
\left|\gamma_x^p\eta_p^w\right|\left|\eta_z^p\gamma_p^y\right|\left|\zeta\right|
\label{d7}\\
&+\sum_{p\in\gamma\cap\eta\cap\operatorname{int}S}
\delta_{x,z}\varepsilon(y;\gamma,\zeta)\varepsilon(p;\gamma,\eta)
\left|\gamma_x^p\eta_p^w\right|\left|\zeta\right|\left|\eta_z^p\gamma_p^y\right|.
\label{d8}
\end{align}
The opposite sign of (\ref{d1}) coincides with a cyclic permutation of $(\gamma,\eta,\zeta)$ at (\ref{d7}). 
Similarly, (\ref{d2}), (\ref{d3}) and (\ref{d4}) correspond to (\ref{d5}), (\ref{d6}) and (\ref{d8}), respectively. 
\begin{align}
\text{($\delta\delta$)}
=&\delta_{x,z}\delta_{x,u}\varepsilon(x;\gamma,\eta)\varepsilon(x;\gamma,\zeta)
\left|\gamma\right|\left|\zeta\right|\left|\eta\right|
\label{dd1}\\
&+\delta_{x,z}\delta_{y,v}\varepsilon(x;\gamma,\eta)\varepsilon(y;\gamma,\zeta)
\left|\zeta\right|\left|\gamma\right|\left|\eta\right|
\label{dd2}\\
&+\delta_{x,z}\delta_{z,u}\varepsilon(x;\gamma,\eta)\varepsilon(z;\eta,\zeta)
\left|\gamma\right|\left|\eta\right|\left|\zeta\right|
\label{dd3}\\
&+\delta_{x,z}\delta_{w,v}\varepsilon(x;\gamma,\eta)\varepsilon(w;\eta,\zeta)
\left|\gamma\right|\left|\zeta\right|\left|\eta\right|
\label{dd4}\\
&+\delta_{y,w}\delta_{x,u}\varepsilon(y;\gamma,\eta)\varepsilon(x;\gamma,\zeta)
\left|\eta\right|\left|\gamma\right|\left|\zeta\right|
\label{dd5}\\
&+\delta_{y,w}\delta_{y,v}\varepsilon(y;\gamma,\eta)\varepsilon(y;\gamma,\zeta)
\left|\eta\right|\left|\zeta\right|\left|\gamma\right|
\label{dd6}\\
&+\delta_{y,w}\delta_{z,u}\varepsilon(y;\gamma,\eta)\varepsilon(z;\eta,\zeta)
\left|\eta\right|\left|\zeta\right|\left|\gamma\right|
\label{dd7}\\
&+\delta_{y,w}\delta_{w,v}\varepsilon(y;\gamma,\eta)\varepsilon(w;\eta,\zeta)
\left|\zeta\right|\left|\eta\right|\left|\gamma\right|.
\label{dd8}
\end{align}
In a similar way, (\ref{dd1}), (\ref{dd2}), (\ref{dd4}) and (\ref{dd6}) correspond to (\ref{dd3}), (\ref{dd7}), (\ref{dd5}) and (\ref{dd8}), respectively. 
We finish proving the Jacobi identity (\ref{aaa}). 
\end{proof}
Lemma~\ref{nabla} and Proposition~\ref{Jacobi} tell us that the pair $(\GSA{\Sigma}^{+}, \nabla)$ is a Poisson algebra. 
We denote it by $\GPA{\Sigma}^{+}$. 
When we put $t=1$, we can obtain the Poisson structure on $\GSA{\Sigma}$. 
We denote it by $\GPA{\Sigma}$.
\begin{RMK}\ 
\begin{enumerate}
\item If $M$ is the empty set, then the Poisson algebra $\GPA{(S,M)}$ is the Goldman Lie algebra of $S$. 
\item Kawazumi-Kuno\cite{Kawazumi-Kuno} defined an action of the Goldman Lie algebra of $S$ on the module generated by the homotopy set of ${\sf Arcs}(\Sigma)$. We can consider the action as the inner derivatoin of $\GPA{\Sigma}$. 
\item Labourie\cite{Labourie2} defined the swapping algebra $\mathcal{L}_k(M)$ where $M$ is a set of points of the boundary of a disk $\mathbb{D}=\{z\in\mathbb{C}\mid \left|z\right|\leq 1\}$ and $k$ any real number. $\GPA{(\mathbb{D},M)}$ coincides with the swapping algebra $\mathcal{L}_0(M)$. 
\end{enumerate}
\end{RMK}

\section{co-Poisson coalgebras of curves on bordered surfaces}\label{sec;coPoisson}
 Let $C_{\mathcal{O}}^{+}$ be the $\mathbf{K}[t,t^{-1}]$-submodule of $\GLA{\Sigma}^{+}$ generated by a trivial loop with no self-intersection. 
The submodule $C_{\mathcal{O}}^{+}$ is uniquely determined although we have two choices of such a trivial loop, the positive or negative curl. 
We denote the quotient $\mathbf{K}[t,t^{-1}]$-module $\GLA{\Sigma}^{+}/(C_{M}^{+}+C_{\mathcal{O}}^{+})$ by $Z(\Sigma)_0^{+}$, 
the quotient map ${\sf Curves}(\Sigma)\to Z(\Sigma)_0^{+}$ by $\left\|\ \cdot\ \right\|$.
The symmetric $\mathbf{K}[t,t^{-1}]$-algebra 
$\GSA{\Sigma}_0^{+}=\operatorname{Sym}(Z(\Sigma)_0^{+})$ has a commutative and associative multiplication $m$ and a cocommutative and coassosiative comultiplication $\Delta$ which define a bialgebra structure on it. 

In this section, we define a cobracket of $\GSA{\Sigma}_0^{+}$ which is a generalization of Turaev cobracket \cite{Turaev1,Turaev2}. 
We show that the cobracket and the bialgebra structure on $\GSA{\Sigma}_0^{+}$ are compatible. 

First, 
we define a cobracket on $\GSA{\Sigma}_0^{+}$. 
For any curve $\alpha$ on $\Sigma$, 
we denote the set of all self-intersection points of $\alpha$ by ${\pitchfork\alpha}$. 
If $\alpha$ is in ${\sf Arcs}(\Sigma)$, 
then ${\pitchfork\alpha}$ decomposes to two subsets type(I) and type(II).
We denote the subset of type(I) (resp. type(II)) self-intersection points by ${\pitchfork_\text{(I)}\alpha}$ (resp. ${\pitchfork_\text{(II)}\alpha}$).
Let $a$ in ${\sf Loops}(\Sigma)$ and $\gamma$ in ${\sf Arcs}(\Sigma)$ be generic. 
We define $\delta'\colon{\sf Curves}(\Sigma)\to\GSA{\Sigma}_0^{+}\otimes_t\GSA{\Sigma}_0^{+}$ as the following.
\begin{align}
\delta'(a)
=&\!\sum_{p\in \pitchfork a}\left\|a_p^{(1)}\right\|\otimes_t \left\|a_p^{(2)}\right\|
-\left\|a_p^{(2)}\right\|\otimes_t\left\|a_p^{(1)}\right\|,\\
\delta'(\gamma)
=&\!\sum_{p\in\pitchfork_{\text{(I)}}\gamma}\left\|\gamma_p^{(1)}\right\|\otimes_t\left\|\gamma_x^p\gamma_p^y\right\|
-\left\|\gamma_x^p\gamma_p^y\right\|\otimes_t\left\|\gamma_p^{(1)}\right\|\\
&+\!\sum_{p\in\pitchfork_{\text{(II)}}\gamma}\left\|\gamma_x^p\gamma_p^y\right\|\otimes_t\left\|\gamma_p^{(2)}\right\|
-\left\|\gamma_p^{(2)}\right\|\otimes_t\left\|\gamma_x^p\gamma_p^y\right\|\nonumber\\
=&\!\sum_{p\in\pitchfork_{\text{(I)}}\gamma\cap\operatorname{int}S}\left\|\gamma_p^{(1)}\right\|\otimes_t\left\|\gamma_x^p\gamma_p^y\right\|
-\left\|\gamma_x^p\gamma_p^y\right\|\otimes_t\left\|\gamma_p^{(1)}\right\|\nonumber\\
&+\!\sum_{p\in\pitchfork_{\text{(II)}}\gamma\cap\operatorname{int}S}\left\|\gamma_x^p\gamma_p^y\right\|\otimes_t\left\|\gamma_p^{(2)}\right\|
-\left\|\gamma_p^{(2)}\right\|\otimes_t\left\|\gamma_x^p\gamma_p^y\right\|.\nonumber
\end{align}
We can confirm that $\delta'$ defines the same element in $\GSA{\Sigma}_0^{+}\otimes_t\GSA{\Sigma}_0^{+}$ for local moves ($\Omega$2-1), ($\Omega$2-2), ($\Omega$3), ($\Omega$b2-2) and ($\Omega$b2-3). 
For example, 
\begin{align*}
\text{the terms coming from}\ \delta'(\alpha)\ \text{of} 
\ \vcenter{\hbox{\includegraphics[width=1cm]{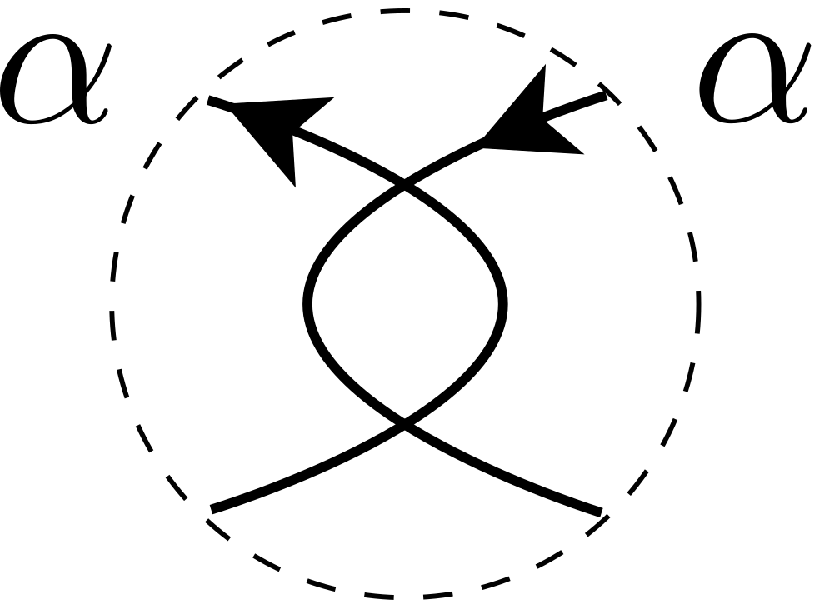}}}
=&\ \vcenter{\hbox{\includegraphics[width=1cm]{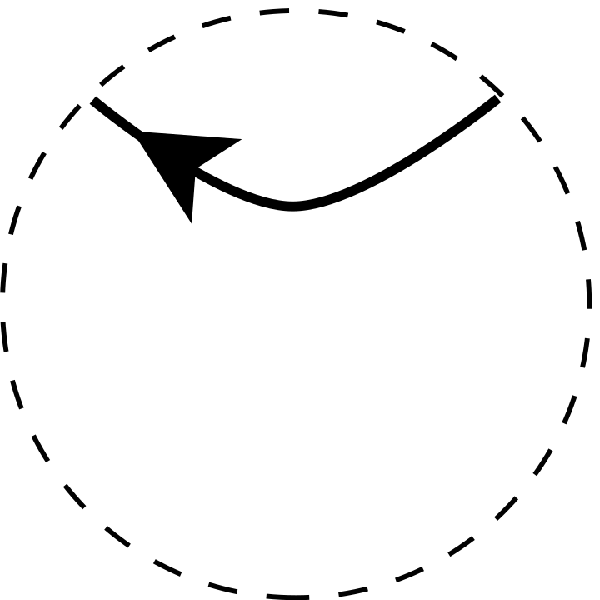}}}
\otimes_t\ \vcenter{\hbox{\includegraphics[width=1cm]{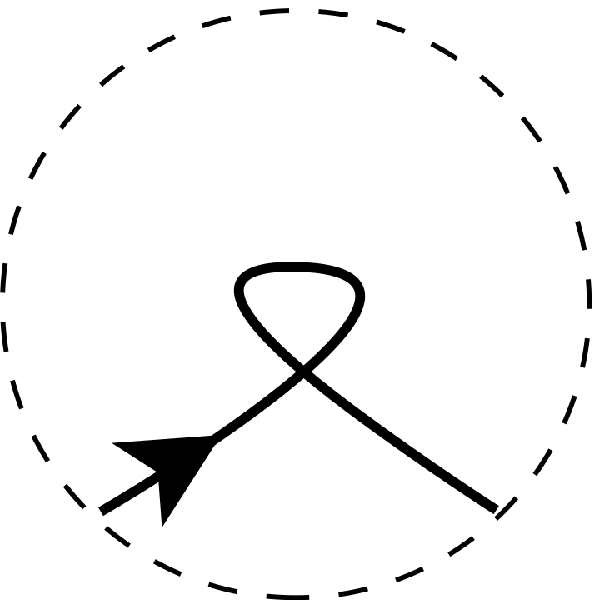}}}
-\ \vcenter{\hbox{\includegraphics[width=1cm]{coinvariance3}}}
\otimes_t\ \vcenter{\hbox{\includegraphics[width=1cm]{coinvariance2}}}\\
&+\ \vcenter{\hbox{\includegraphics[width=1cm]{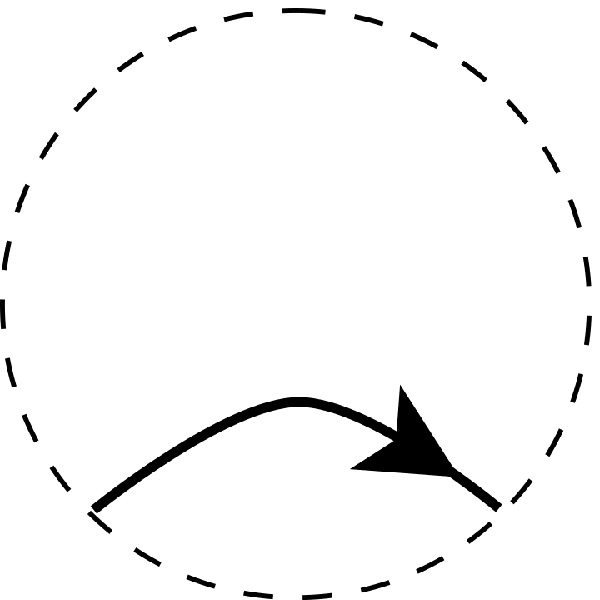}}}
\otimes_t\ \vcenter{\hbox{\includegraphics[width=1cm]{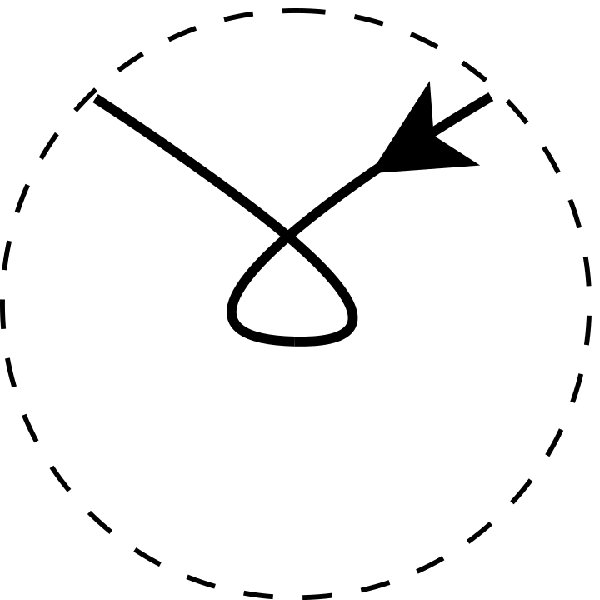}}}
-\ \vcenter{\hbox{\includegraphics[width=1cm]{coinvariance3180}}}
\otimes_t\ \vcenter{\hbox{\includegraphics[width=1cm]{coinvariance2180}}}\\
=&\ \vcenter{\hbox{\includegraphics[width=1cm]{coinvariance2}}}
\otimes_t t\ \vcenter{\hbox{\includegraphics[width=1cm]{coinvariance2180}}}
-t\ \vcenter{\hbox{\includegraphics[width=1cm]{coinvariance2180}}}
\otimes_t\ \vcenter{\hbox{\includegraphics[width=1cm]{coinvariance2}}}\\
&+\ \vcenter{\hbox{\includegraphics[width=1cm]{coinvariance2180}}}
\otimes_t t\ \vcenter{\hbox{\includegraphics[width=1cm]{coinvariance2}}}
-t\ \vcenter{\hbox{\includegraphics[width=1cm]{coinvariance2}}}
\otimes_t\ \vcenter{\hbox{\includegraphics[width=1cm]{coinvariance2180}}}\\
=&0.
\end{align*}
We do not need to calculate local moves ($\Omega$b2-1) and ($\Omega$b2-4) because connected arcs have no such self-intersections. 
If $\alpha$ is a curl or a monogon with vertex in $M$, then the value $\delta'(\alpha)$ vanishes.
Therefore the map $\delta'$ extends to $\delta\colon Z(\Sigma)^{+}\to\GSA{\Sigma}_0^{+}\otimes_t\GSA{\Sigma}_0^{+}$ and the extended map $\delta$ is a $\mathbf{K}[t,t^{-1}]$-linear map. 
Furthermore, we extend it to the $\mathbf{K}[t,t^{-1}]$-linear map on $\GSA{\Sigma}^{+}$ by using the $(m,\delta)$-compatibility. (See Appendix.)
Consequently, we obtain the $\mathbf{K}[t,t^{-1}]$-linear map $\delta\colon\GSA{\Sigma}_0^{+}\to\GSA{\Sigma}_0^{+}\otimes_t\GSA{\Sigma}_0^{+}$. 
\begin{LEM}\label{delta}\ 
\begin{enumerate}
\item If $\delta'$ is co-skew symmetric, 
then $\delta$ is co-skew symmetric.
\item If $\delta'$ satisfies the co-Leibniz rule, 
then $\delta$ satisfies the co-Leibniz rule.
\item If $\delta'$ is co-skew symmetric and satisfies the co-Leibniz rule and the co-Jacobi identity, 
then $\delta$ satisfies the co-Jacobi identity.
\end{enumerate}
\begin{proof}
Use the $(m,\delta)$-compatibility and an induction on the degree of symmetric power.
\end{proof}
\end{LEM}
\begin{PROP}\label{co-Jacobi}
$\delta'$ is co-skew symmetric and satisfies the co-Jacobi identity.
\end{PROP}
\begin{proof}
It is clear that the map $\delta'$ is co-skew symmetric by definition.
We directly calculate $(\tau^2+\tau+\textup{id}^{\otimes_t 3})\circ(\textup{id}\otimes_t\delta')\circ\delta'(\alpha)$ to prove the co-Jacobi identity.
If $\alpha$ is a loop on $\Sigma$, 
then the map $\delta'$ coincides with the Turaev cobracket.
We have to confirm the co-Jacobi identity when $\alpha$ is an arc on $\Sigma$.

Let $\gamma=\gamma_x^y$ is an arc on $\Sigma$ with endpoints $x$ and $y$ in $M$.
Then, 
\begin{align}
&(\textup{id}\otimes_t\delta')\circ\delta'(\gamma_x^y)\nonumber\\
\begin{split}
&\quad=\sum_{p\in{\pitchfork_{\text{(I)}}\gamma}}\sum_{q\in{\pitchfork_{\text{(I)}}\gamma_x^p\gamma_p^y}}
\Big(\left\|\gamma_p^{(1)}\right\|
\otimes_t\left\|(\gamma_x^p\gamma_p^y)_q^{(1)}\right\|
\otimes_t\left\|(\gamma_x^p\gamma_p^y)_x^q(\gamma_x^p\gamma_p^y)_q^y\right\|\\
&\qquad\qquad-\left\|\gamma_p^{(1)}\right\|
\otimes_t\left\|(\gamma_x^p\gamma_p^y)_x^q(\gamma_x^p\gamma_p^y)_q^y\right\|
\otimes_t\left\|(\gamma_x^p\gamma_p^y)_q^{(1)}\right\|\Big)
\end{split}\label{coJac11}\\
\begin{split}
&\quad+\sum_{p\in{\pitchfork_{\text{(I)}}\gamma}}\sum_{q\in{\pitchfork_{\text{(II)}}\gamma_x^p\gamma_p^y}}
\Big(\left\|\gamma_p^{(1)}\right\|
\otimes_t\left\|(\gamma_x^p\gamma_p^y)_x^q(\gamma_x^p\gamma_p^y)_q^y\right\|
\otimes_t\left\|\gamma_p^{(2)}\right\|\\
&\qquad\qquad-\left\|\gamma_p^{(1)}\right\|
\otimes_t\left\|\gamma_p^{(2)}\right\|
\otimes_t\left\|(\gamma_x^p\gamma_p^y)_x^q(\gamma_x^p\gamma_p^y)_q^y\right\|\Big)
\end{split}\label{coJac12}\\
\begin{split}
&\quad-\sum_{p\in{\pitchfork_{\text{(I)}}\gamma}}\sum_{q\in{\pitchfork\gamma_p^{(1)}}}
\Big(\left\|\gamma_x^p\gamma_p^y\right\|
\otimes_t\left\|(\gamma_p^{(1)})_q^{(1)}\right\|
\otimes_t\left\|(\gamma_p^{(1)})_q^{2}\right\|\\
&\qquad\qquad-\left\|\gamma_x^p\gamma_p^y\right\|
\otimes_t\left\|(\gamma_p^{(1)})_q^{2}\right\|
\otimes_t\left\|(\gamma_p^{(1)})_q^{(1)}\right\|\Big)
\end{split}\label{coJac10}\\
\begin{split}
&\quad+\sum_{p\in{\pitchfork_{\text{(II)}}\gamma}}\sum_{q\in{\pitchfork\gamma_p^{(2)}}}
\Big(\left\|\gamma_x^p\gamma_p^y\right\|
\otimes_t\left\|(\gamma_p^{(2)})_q^{(1)}\right\|
\otimes_t\left\|(\gamma_p^{(2)})_q^{2}\right\|\\
&\qquad\qquad-\left\|\gamma_x^p\gamma_p^y\right\|
\otimes_t\left\|(\gamma_p^{(2)})_q^{2}\right\|
\otimes_t\left\|(\gamma_p^{(2)})_q^{(1)}\right\|\Big)
\end{split}\label{coJac20}\\
\begin{split}
&\quad-\sum_{p\in{\pitchfork_{\text{(II)}}\gamma}}\sum_{q\in{\pitchfork_{\text{(I)}}\gamma_x^p\gamma_p^y}}
\Big(\left\|\gamma_p^{(2)}\right\|
\otimes_t\left\|(\gamma_x^p\gamma_p^y)_q^{(1)}\right\|
\otimes_t\left\|(\gamma_x^p\gamma_p^y)_x^q(\gamma_x^p\gamma_p^y)_q^y\right\|\\
&\qquad\qquad-\left\|\gamma_p^{(1)}\right\|
\otimes_t\left\|(\gamma_x^p\gamma_p^y)_x^q(\gamma_x^p\gamma_p^y)_q^y\right\|
\otimes_t\left\|(\gamma_x^p\gamma_p^y)_q^{(1)}\right\|\Big)
\end{split}\label{coJac21}\\
\begin{split}
&\quad-\sum_{p\in{\pitchfork_{\text{(II)}}\gamma}}\sum_{q\in{\pitchfork_{\text{(II)}}\gamma_x^p\gamma_p^y}}
\Big(\left\|\gamma_p^{(2)}\right\|
\otimes_t\left\|(\gamma_x^p\gamma_p^y)_x^q(\gamma_x^p\gamma_p^y)_q^y\right\|
\otimes_t\left\|\gamma_p^{(2)}\right\|\\
&\qquad\qquad-\left\|\gamma_p^{(1)}\right\|
\otimes_t\left\|\gamma_p^{(2)}\right\|
\otimes_t\left\|(\gamma_x^p\gamma_p^y)_x^q(\gamma_x^p\gamma_p^y)_q^y\right\|\Big)
\end{split}\label{coJac22}
\end{align}
where ${\pitchfork\gamma_p^{(i)}}$ in (\ref{coJac10}) and (\ref{coJac20}) is the set of self-intersection points of the free loop obtained by $\gamma_p^{(i)}$ for $i=1,2$.
Setting 
$\gamma_x^p\cap_\text{(I)}\gamma_p^y={\pitchfork_\text{(I)}\gamma_x^y}\cap\gamma_x^p\cap\gamma_p^y\setminus\{p\}$, 
we can obtain the following decomposition
$\pitchfork_\text{(I)}\gamma_x^p\gamma_p^y
=({\pitchfork_\text{(I)}\gamma_x^p})\cup({\pitchfork_\text{(I)}\gamma_p^y})\cup(\gamma_x^p\cap_\text{(I)}\gamma_p^y)$.
By using the decomposition,
\begin{align*}
\begin{split}
(\ref{coJac11})&=\sum_{p\in{\pitchfork_{\text{(I)}}\gamma}}\sum_{q\in{\pitchfork_{\text{(I)}}\gamma_x^p}}
\Big(\left\|\gamma_p^{(1)}\right\|
\otimes_t\left\|\gamma_q^{(1)}\right\|
\otimes_t\left\|\gamma_x^q\gamma_q^p\gamma_p^y\right\|
-\left\|\gamma_p^{(1)}\right\|
\otimes_t\left\|\gamma_x^q\gamma_q^p\gamma_p^y\right\|
\otimes_t\left\|\gamma_q^{(1)}\right\|\Big)
\end{split}\\
\begin{split}
&\quad+\sum_{p\in{\pitchfork_{\text{(I)}}\gamma}}\sum_{q\in{\pitchfork_{\text{(I)}}\gamma_p^y}}
\Big(\left\|\gamma_p^{(1)}\right\|
\otimes_t\left\|\gamma_q^{(1)}\right\|
\otimes_t\left\|\gamma_x^p\gamma_p^q\gamma_q^y\right\|
-\left\|\gamma_p^{(1)}\right\|
\otimes_t\left\|\gamma_x^p\gamma_p^q\gamma_q^y\right\|
\otimes_t\left\|\gamma_q^{(1)}\right\|\Big)
\end{split}\\
\begin{split}
&\quad+\sum_{p\in{\pitchfork_{\text{(I)}}\gamma}}\sum_{q\in{\gamma_x^p\cap_\text{(I)}\gamma_p^y}}
\Big(\left\|\gamma_p^{(1)}\right\|
\otimes_t\left\|\gamma_q^{(1)}\right\|
\otimes_t\left\|\gamma_q^p\gamma_p^q\right\|
-\left\|\gamma_x^q\gamma_q^y\right\|
\otimes_t\left\|\gamma_x^q\gamma_q^y\right\|
\otimes_t\left\|\gamma_q^p\gamma_p^q\right\|\Big).
\end{split}
\end{align*}
We calculate (\ref{coJac12})--(\ref{coJac22}) in the same way and we can confirm the co-Jacobi identity.
\end{proof}
By Lemma~\ref{delta} and Proposition~\ref{co-Jacobi}, the pair $(\GSA{\Sigma}_0^{+},\delta)$ is a co-Poisson coalgebra. 
We denote it by $\GCPA{\Sigma}^{+}$. 
\begin{RMK}
We can define the co-Poisson coalgebra $\GCPA{\Sigma}$ by substituting $t=1$.
To explain more precisely, 
we denote by $C_{\mathcal{O}}$ the $\mathbf{K}$-submodule of $\GLA{\Sigma}$ generated by the trivial loop.
The quotient $Z(\Sigma)_0=\GLA{\Sigma}/(C_M+C_{\mathcal{O}})$ is $Z(\Sigma)_0^{+}/(t-1)Z(\Sigma)_0^{+}$. 
Then we can define the cobracket on the symmetric $\mathbf{K}$-algebra $\GSA{\Sigma}_0=\operatorname{Sum}(Z(\Sigma)_0)$ which induced from the above $\delta$. 
We denote the co-Poisson coalgebra by $\GCPA{\Sigma}$.
\end{RMK}
\begin{RMK}
The bracket on $\GPA{\Sigma}^{+}$ induces a Poisson structure on $\GSA{\Sigma}_0^{+}$. 
Although $\GSA{\Sigma}_0^{+}$ has the bracket and the cobracket, 
they do not satisfy the compatibility conditions of a bi-Poisson bialgebra.
To be precise, the $(\nabla,\Delta)$-compatibility is not satisfied for two arcs on $\Sigma$.
We show a counter-example in Figure~\ref{counterex}.
\end{RMK}
\begin{figure}[h]
\centering
\includegraphics[scale=0.5, clip]{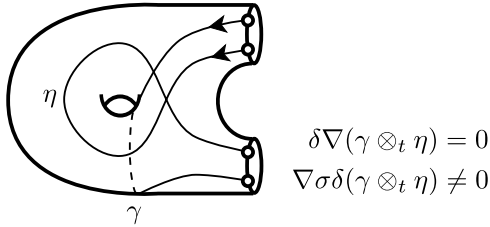}
\caption{a counter-example of the $(\nabla,\Delta)$-compatibility}
\label{counterex}
\end{figure}

\section{Poisson algebras of unoriented curves and its quantization}\label{sec;quantization}
Let $\Sigma=(S,M)$ be a bordered surface. 
We define unoriented curves on $\Sigma$ by identifying oriented curves on $\Sigma$ and its inverse. 
We denote the set of unoriented curves, loops and arcs 
by $\overline{{\sf Curves}}(\Sigma)$, $\overline{{\sf Loops}}(\Sigma)$ and $\overline{{\sf Arcs}}(\Sigma)$ respectively. 
We denote the set of regular homotopy classes of unoriented curves on $\Sigma$ 
by $\bar{\pi}(\Sigma)^{+}$. 
Let $G=\langle\ T\mid T^2=1\ \rangle$ be the cyclic group of order $2$, 
$\mathbf{K}G$ the group ring and $\UGLA{\Sigma}^{+}$ the free $\mathbf{K}G$-module with basis $\bar{\pi}(\Sigma)^{+}$. 
The submodule $\bar{C}_T$ of $\UGLA{\Sigma}^{+}$ is generated by the elements $T\left|\alpha\right|-\left|\alpha'\right|$ for any $\alpha$ in $\overline{{\sf Curves}}(\Sigma)$ 
where $\alpha'$ is a curve obtained by inserting a monogon into interior of $\alpha$. 
We remark that we have two choices of $\alpha'$ because $\alpha$ is two-sided and these elements in $\bar{\pi}(\Sigma)^{+}$ may be different. 
The submodule $\bar{C}_M^{+}$ of $\UGLA{\Sigma}^{+}$ is generated by the unoriented monogons with vertices in $M$.
Then the $\mathbf{K}G$-module $\bar{Z}(\Sigma)^{+}$ is the quotient of $\UGLA{\Sigma}^{+}$ 
by $\bar{C}_T+\bar{C}_M^{+}$ and $\UGPA{\Sigma}^{+}=\operatorname{Sym}(\bar{Z}(\Sigma)^{+})$.
We use the same symbol $\left|\ \cdot\ \right|$ 
for representing the quotient map $\overline{{\sf Curves}}(\Sigma)\to\bar{Z}(\Sigma)^{+}$. 
For any subset $X=\{a_1,a_2,\dots,a_n\}$ of $\overline{{\sf Curves}}(\Sigma)$, 
$\left|X\right|$ denotes the element $\left|a_1\right|\left|a_2\right|\cdots\left|a_n\right|$ of $\UGPA{\Sigma}^{+}$.

Let $\{a_1,a_2,\dots,a_n\}$ be a generic subset of $\overline{{\sf Curves}}(\Sigma)$. 
We define local moves of $a_1\cup a_2\cup\dots\cup a_n$ for $n$-tuple of generic unoriented curves $(a_1,a_2,\dots,a_n)$. 
Let $p$ be an intersection point of $a_i$ and $a_j$ where $1\leq i < j\leq n$. 
A neighborhood of $p$ has two subarcs $a_i'$ of $a_i$ and $a_j'$ of $a_j$.
We give an orientation of the subarcs 
such that the pair of tangent vectors at $p$ obtained from $(a_i',a_j')$ coincides with the orientation of the surface $S$. 
$\mathsf{E}_p(a_1,a_2,\dots,a_n)$ denotes unoriented curves obtained by smoothing the crossing at $p$ under the given orientation. 
We define $\mathsf{e}_p(a_1,a_2,\dots,a_n)$ as the unoriented curves obtained by exchanging $a_i$ and $a_j$ of $\mathsf{E}_p(a_1,a_2,\dots,a_n)$. 
We illustrate the two eliminations $\mathsf{E}_p$ and $\mathsf{e}_p$ at the intersection point $p$ in Figure~\ref{elimination}. 
\begin{figure}[h]
\centering
\includegraphics[scale=0.5, clip]{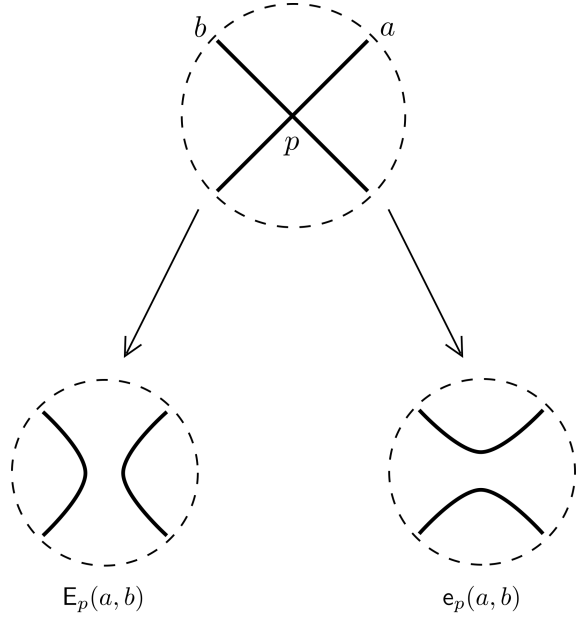}
\caption{elimination of crossing at $p$}
\label{elimination}
\end{figure}
Let $p_1,p_2,\dots,p_m$ be intersection points of distinct two curves in $\bigcup_{1\leq i\leq n} a_i$. 
Let us denote by $\mathsf{X}_{p_1}\mathsf{X}_{p_2}\cdots\mathsf{X}_{p_m}(a_1,a_2,\dots,a_n)$ unoriented curves obtained 
by performing the eliminations $\mathsf{X}_{p_i}$ under the above rule 
where the symbol $\mathsf{X}$ is $\mathsf{E}$ or $\mathsf{e}$. 

For any $a$ and $b$ in $\overline{{\sf Curves}}(\Sigma)$, 
We define a bracket on $\UGSA{\Sigma}^{+}$ by 
\begin{align}\label{unoriPoisson}
\begin{split}
 \{a,b\}
=&\sum_{p\in a\cap b \cap \operatorname{int}S}\left|\mathsf{E}_p(a,b)\right|-\left|\mathsf{e}_p(a,b)\right|\\
&\qquad+\frac{1}{2}(n_{+}(a,b)-n_{-}(a,b))\left|a\right|\left|b\right|.
\end{split}
\end{align}
In the above, $n_{+}(a,b)$ (resp. $n_{-}(a,b)$) is the number of pairs of endpoints of $a$ and $b$ 
such that the pairs of endpoints lie in the same marked point $m$ in $M$ 
and $a$ is locally located on the right (resp. left) side of $b$ in the neighborhood of $m$. 
We can confirm that the value of the bracket is invariant under the unoriented local moves ($\bar{\Omega}$1)--($\bar{\Omega}$3) and ($\bar{\Omega}$b2) illustrated in Figure~\ref{unorihomotopymoves}.
\begin{figure}[h]
\centering
\includegraphics[scale=0.5, clip]{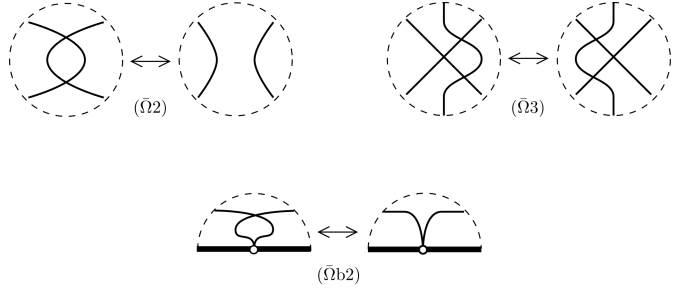}
\caption{the local moves of unoriented curves}
\label{unorihomotopymoves}
\end{figure}
Thus the bracket is defined on $\hat{\pi}(\Sigma)^{+}$ and we extend it $\mathbf{K}G$-bilinearly. 
We obtain $\{\cdot,\cdot\}\colon\UGLA{\Sigma}\otimes_G\UGLA{\Sigma}\to\UGSA{\Sigma}^{+}$ 
where the symbol $\otimes_G$ denotes the tensor product over $\mathbf{K}G$. 
If $a$ or $b$ is a monogon with vertex in $M$, then the bracket vanishes. 
We can also confirm that the bracket vanishes on 
$\UGLA{\Sigma}\otimes_G\bar{C}_T+\bar{C}_T\otimes_G\UGLA{\Sigma}$
because the bracket preserves monogons in the curves. 
Therefore, we can define the bracket on $\bar{Z}(\Sigma)^{+}$ and extend to $\UGSA{\Sigma}^{+}$ by the Leibniz rule. 

We can easily show the following lemma by definition. 
\begin{LEM}\label{unorilem}
Let $\{a_1,a_2,\dots,a_n\}$ and $\{a,b\}$ be generic subsets of $\overline{{\sf Curves}}(\Sigma)$. 
\begin{enumerate}
\item $\mathsf{E}_p(a_1,a_2,\dots,a_i,\dots,a_j,\dots,a_n)=\mathsf{e}_p(a_1,a_2,\dots,a_j,\dots,a_i,\dots,a_n)$ for an intersection point $p$ of $a_i$ and $a_j$.
\item $n_{+}(a,b)=n_{-}(a,b)$.
\end{enumerate}
\end{LEM}
\begin{PROP}\label{unoriLie}
The pair $(\UGSA{\Sigma}^{+},\{\cdot,\cdot\})$ is a Poisson algebra. 
\end{PROP}
\begin{proof}
Lemma~\ref{unorilem} implies skew-symmetry of the bracket. 
We have to show the Jacobi identity. 
Let us denote $n_{+}(a,b)-n_{-}(a,b)$ by $N(a,b)$. 
Let $a,b$ and $c$ be in $\overline{{\sf Curves}}(\Sigma)$. 
\begin{align}
\{\{\left|a\right|,\left|b\right|\},\left|c\right|\}
=&\sum_{p\in a\cap b\cap\operatorname{int}S}
\{\left|\mathsf{E}_p(a,b)\right|,\left|c\right|\}-\{\left|\mathsf{e}_p(a,b)\right|,\left|c\right|\}\nonumber\\
&+\frac{1}{2}N(a,b)\{\left|a\right|\left|b\right|,\left|c\right|\}.\nonumber
\end{align}
For any $p$ in $a\cap b\cap\operatorname{int}S$, 
\begin{align}
\{\left|\mathsf{e}_p(a,b)\right|,\left|c\right|\}
=&\sum_{q\in a\cap c\cap\operatorname{int}S}
\left|\mathsf{E}_q\mathsf{e}_p(a,b,c)\right|-\left|\mathsf{e}_q\mathsf{e}_p(a,b,c)\right|\nonumber\\
&+\sum_{q\in b\cap c\cap\operatorname{int}S}
\left|\mathsf{E}_q\mathsf{e}_p(a,b,c)\right|-\left|\mathsf{e}_q\mathsf{e}_p(a,b,c)\right|\nonumber\\
&+\frac{1}{2}(N(a,c)+N(b,c))\left|\mathsf{E}_p(a,b)\right|\left|c\right|,\nonumber
\end{align}
\begin{align}
\{\left|\mathsf{E}_p(a,b)\right|,\left|c\right|\}
=&\sum_{q\in a\cap c\cap\operatorname{int}S}
\left|\mathsf{E}_q\mathsf{E}_p(a,b,c)\right|-\left|\mathsf{e}_q\mathsf{E}_p(a,b,c)\right|\nonumber\\
&+\sum_{q\in b\cap c\cap\operatorname{int}S}
\left|\mathsf{E}_q\mathsf{E}_p(a,b,c)\right|-\left|\mathsf{e}_q\mathsf{E}_p(a,b,c)\right|\nonumber\\
&+\frac{1}{2}(N(a,c)+N(b,c))\left|\mathsf{e}_p(a,b)\right|\left|c\right|.\nonumber
\end{align}
\begin{align}
\{\left|a\right|\left|b\right|,\left|c\right|\}
=&\sum_{q\in a\cap c\cap\operatorname{int}S}
\left|\mathsf{E}_q(a,c)\right|\left|b\right|-\left|\mathsf{e}_q(a,c)\right|\left|b\right|\nonumber\\
&+\sum_{q\in b\cap c\cap\operatorname{int}S}
\left|\mathsf{E}_q(b,c)\right|\left|a\right|-\left|\mathsf{e}_q(b,c)\right|\left|a\right|\nonumber\\
&+\frac{1}{2}(N(a,c)+N(b,c))\left|a\right|\left|b\right|\left|c\right|.\nonumber
\end{align}
Taking the above calculations together, 
\begin{align}
&\{\{\left|a\right|,\left|b\right|\},\left|c\right|\}\nonumber\\
&=\sum_{p\in a\cap b\cap\operatorname{int}S}\sum_{q\in a\cap c\cap\operatorname{int}S}
\left|\mathsf{E}_q\mathsf{E}_p(a,b,c)\right|-\left|\mathsf{e}_q\mathsf{E}_p(a,b,c)\right|
-\left|\mathsf{E}_q\mathsf{e}_p(a,b,c)\right|+\left|\mathsf{e}_q\mathsf{e}_p(a,b,c)\right|\label{intint1}\\
&+\sum_{p\in a\cap b\cap\operatorname{int}S}\sum_{q\in b\cap c\cap\operatorname{int}S}
\left|\mathsf{E}_q\mathsf{E}_p(a,b,c)\right|-\left|\mathsf{e}_q\mathsf{E}_p(a,b,c)\right|
-\left|\mathsf{E}_q\mathsf{e}_p(a,b,c)\right|+\left|\mathsf{e}_q\mathsf{e}_p(a,b,c)\right|\label{intint2}\\
&+\sum_{p\in a\cap b\cap\operatorname{int}S}
\frac{1}{2}(N(a,c)+N(b,c))
(\left|\mathsf{E}_p(a,b)\right|\left|c\right|-\left|\mathsf{e}_p(a,b)\right|\left|c\right|)\label{int1}\\
&+\sum_{p\in a\cap c\cap\operatorname{int}S}
\frac{1}{2}N(a,b)
(\left|\mathsf{E}_q(a,c)\right|\left|b\right|-\left|\mathsf{e}_q(a,c)\right|\left|b\right|)\label{int2}\\
&+\sum_{p\in b\cap c\cap\operatorname{int}S}
\frac{1}{2}N(a,b)
(\left|\mathsf{E}_q(b,c)\right|\left|a\right|-\left|\mathsf{e}_q(b,c)\right|\left|a\right|)\label{int3}\\
&+\frac{1}{4}N(a,b)(N(a,c)-N(b,c))\left|a\right|\left|b\right|\left|c\right|.\label{mark}
\end{align}
Replacing the symbols $(a,b,c)$ of (\ref{intint1}) with $(b,c,a)$, 
\begin{align}
&\sum_{p\in b\cap c\cap\operatorname{int}S}\sum_{q\in b\cap a\cap\operatorname{int}S}
\left|\mathsf{E}_q\mathsf{E}_p(b,c,a)\right|-\left|\mathsf{e}_q\mathsf{E}_p(b,c,a)\right|
-\left|\mathsf{E}_q\mathsf{e}_p(b,c,a)\right|+\left|\mathsf{e}_q\mathsf{e}_p(b,c,a)\right|\nonumber\\
&=\sum_{q\in a\cap b\cap\operatorname{int}S}\sum_{p\in b\cap c\cap\operatorname{int}S}
\left|\mathsf{e}_q\mathsf{E}_p(a,b,c)\right|-\left|\mathsf{E}_q\mathsf{E}_p(a,b,c)\right|
-\left|\mathsf{e}_q\mathsf{e}_p(a,b,c)\right|+\left|\mathsf{E}_q\mathsf{e}_p(b,c,a)\right|\nonumber\\
&=-(\ref{intint2})\nonumber
\end{align}
by using Lemma~\ref{unorilem}. 
Consequently, (\ref{intint1}) and (\ref{intint2}) vanish in 
$\{a,\{b,c\}\}+\{b,\{c,a\}\}+\{c,\{a,b\}\}$. 
We can show that (\ref{int1}) of $\{b,\{c,a\}\}$ cancels the sum of (\ref{int3}) and (\ref{int2}) of $\{c,\{a,b\}\}$ by similar calculation. 
Finally, We can confirm that (\ref{mark}) of $\{a,\{b,c\}\}+\{b,\{c,a\}\}+\{c,\{a,b\}\}$ vanishes by using $N(a,b)=-N(b,a)$. 
\end{proof}
We denote the Poisson algebra $(\UGSA{\Sigma}^{+},\{\cdot,\cdot\})$ by $\UGPA{\Sigma}^{+}$. 

We can consider variations of $\UGPA{\Sigma}^{+}$ by substituting $T=1$ and $T=-1$ by the same way as substituting $t=1$ in $\GPA{\Sigma}^{+}$ and $\GCPA{\Sigma}^{+}$. 
Let us denote these Poisson $\mathbf{K}$-algebras by $\UGPA{\Sigma}^{+}|_{T=1}$ and $\UGPA{\Sigma}^{+}|_{T=-1}$. 
If we substitute $T=1$, then $\mathbf{K}G$ is replaced with $K$, 
$\hat{Z}^{+}(\Sigma)$ with $\hat{Z}(\Sigma)=\mathbf{K}\hat{\pi}(\Sigma)/\bar{C}_M$ where $\hat{\pi}(\Sigma)$ is the set of homotopy classes of curves on $\Sigma$ and $\bar{C}_M$ is the submodule of $\mathbf{K}\hat{\pi}(\Sigma)$ generated by the contractible arcs.
Therefore $\UGPA{\Sigma}^{+}|_{T=1}$ is the symmetric $\mathbf{K}$-algebra of $\bar{Z}(\Sigma)$ with the Poisson bracket induced from $\UGPA{\Sigma}^{+}$. 
We denote the Poisson algebra by $\UGPA{\Sigma}$. 

We will discuss about the Poisson algebra $\UGPA{\Sigma}^{+}|_{T=-1}$.
Concretely speaking, we will give a quantization of a quotient Poisson algebra of $\UGPA{\Sigma}^{+}|_{T=-1}$. 
As a first step, we define the skein algebra of $\Sigma$ defined by Muller~\cite{Muller}. 
The skein algebra $\SKQ{\Sigma}$ is given by the quotient of the algebra of framed links in $\Sigma\times [0,1]$ by the skein relators. 
Next, we give the Poisson algebra $\SK{\Sigma}$ which is a quotient of $\UGPA{\Sigma}^{+}|_{T=-1}$ by the relations corresponding to the skein relations. 
Finally, we will give a quantization homomorphism from $\SKQ{\Sigma}$ to $\SK{\Sigma}$ by $q\to 1$. 

We give the definition of framed links on a bordered surface $\Sigma$ combinatorially by using link diagrams according to Muller~\cite{Muller}.
\begin{DEF}\label{diagram}
Let $D=\{a_1,a_2,\dots,a_n\}\subset\overline{\mathsf{Curves}}(\Sigma)$ be generic unoriented curves on $\Sigma$.
We consider $D$ as the immersion $a_1\cup a_2\cup\dots\cup a_n$ from the disjoint union of domains to $S$.
We give a strict order ``<'' on $D^{-1}(p)$ for any intersection point $p$ of $D$ in $\operatorname{int}S$. 
We give a total order ``$\leq$'' on $D^{-1}(m)$ for any intersection point $m$ of $D$ in $M$.
We define \emph{the link diagram} $D$ on $\Sigma$ as generic curves $D$ equipped with these ordered set, \emph{the crossing data} of $D$.
\end{DEF}
If the strictly ordered set $D^{-1}(p)=\{p_1,p_2\}$ has the order $p_1<p_2$, then we can also describe diagrammatically it in a neighborhood of $p\in\operatorname{int}S$ as Figure~\ref{crossing}.
We illustrate the neighborhood of $p\in\operatorname{int}S$ by gapping the strand containing $p_2$ at $p$. 
We call the strand containing $p_1$ (resp. $p_2$) is \emph{the overcrossing} (resp. \emph{the undercrossing}) of $p_1$ (resp. $p_2$) at $p$.

If we take two elements $m_1$ and $m_2$ from the total ordered set $D^{-1}(m)$,
we can also describe the strand of $m_1$ and $m_2$ diagrammatically in a neighborhood of $m$ in the same way. 
We illustrate as the real crossing for $m_1=m_2$.
We call the strand of $m_1$ is \emph{the simultaneous-crossing} with $m_2$ at $m$.
\begin{figure}[h]
\centering
\includegraphics[scale=0.5, clip]{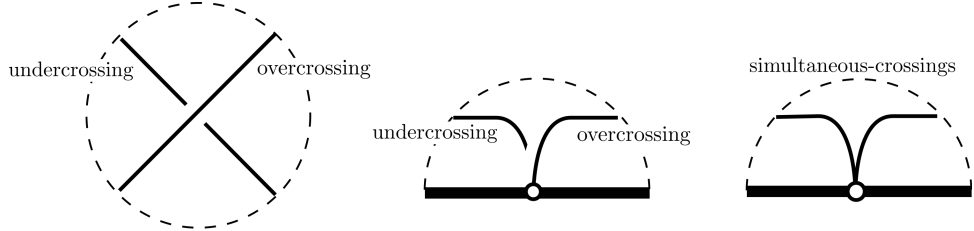}
\caption{the crossing data}
\label{bcrossing}
\end{figure}

Let us denote by $\mathscr{D}(\Sigma)$ the set of link diagrams on $\Sigma$.
We define an equivalence relation, \emph{regular isotopy}, on $\mathscr{D}(\Sigma)$ according to the approach of Kauffman~\cite{Kauffman}.
\begin{DEF}\label{framedlink}
Link diagrams $D$ and $D'$ are \emph{regularly isotopic} if $D$ is obtained from $D'$ by a finite sequence of ambient isotopy of $S$ fixing collar of $\partial S$ and preserving crossing data and Reidemeister moves (R2), (R3), (Rb2-1) and (Rb2-2). See Figure~\ref{Reidemeister}.
The set of regular isotopy classes of link diagrams on $\Sigma$ is denoted by $\mathscr{L}(\Sigma)^{+}$.
A \emph{framed link} is an element in $\mathscr{L}(\Sigma)^{+}$ and we denote by $[D]$ the framed link represented by the link diagram $D$.
\end{DEF}
\begin{figure}[h]
\centering
\includegraphics[scale=0.5, clip]{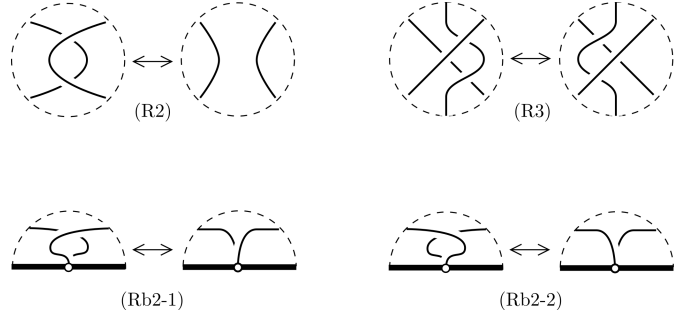}
\caption{the Reidemeister moves}
\label{Reidemeister}
\end{figure}
We remark that $\mathscr{D}(\Sigma)$ has the empty diagram $\emptyset$ and $\mathscr{L}(\Sigma)^{+}$ has the empty link $\left[\emptyset\right]$.
We can define the multiplication of two link diagrams $D_1$ and $D_2$ by superposition of link diagrams if $D_1\cup D_2$ is generic.
Then we give crossing data on $D_1\cap D_2$ such that all strands of $D_1$ contained in neighborhoods of $D_1\cap D_2$ are overcrossings and strands of $D_2$ are undercrossings.
We denote the multiplication by $D_1\cdot D_2$.
Let $L_1$ and $L_2$ be framed links on $\Sigma$.
We can take representatives $D_1$ of $L_1$ and $D_2$ of $L_2$ such that $D_1\cup D_2$ is generic curves on $\Sigma$.
Then $\mathscr{L}(\Sigma)^{+}$ has the multiplicative structure defined by $L_1L_2=\left[D_1\cdot D_2\right]$.

Next, we introduce the skein algebra of $\Sigma$.
Let us denote by $\mathbf{K}[q^{\frac{1}{2}}, q^{-\frac{1}{2}}]$ the ring of Laurent polynomials in one variable $q^{\frac{1}{2}}$ and 
$\mathbf{K}[q^{\frac{1}{2}}, q^{-\frac{1}{2}}]\mathscr{L}(\Sigma)^{+}$ the free $\mathbf{K}[q^{\frac{1}{2}}, q^{-\frac{1}{2}}]$-module on $\mathscr{L}(\Sigma)^{+}$.
The free module $\mathbf{K}[q^{\frac{1}{2}}, q^{-\frac{1}{2}}]\mathscr{L}(\Sigma)^{+}$ has the multiplication induced from $\mathscr{L}(\Sigma)^{+}$.
We prepare two sets $r(b)$ and $\ell(b)$ to define a relation of the skein algebra for any strand $b$ contained in a neighborhood of $M$ of a link diagram $D$.
The elements of $r(b)$ (resp. $\ell(b)$) is the simultaneous-crossings with $b$ such that these simultaneous-crossings are located on the right (resp. left) side of the simultaneous-crossing of $b$.

Let $D$ be a link diagram.
For an intersection point $p\in\operatorname{int}S$ of $D$, we define link diagrams $D_0^p$ and $D_\infty^p$ by eliminating the crossing at $p$ as illustrated in Figure~\ref{Kauffman}.
The crossing data on the intersection points of $D_0$ and $D_\infty$ are naturally induced from the crossing data of $D$.
\begin{figure}[h]
\centering
\includegraphics[scale=0.5, clip]{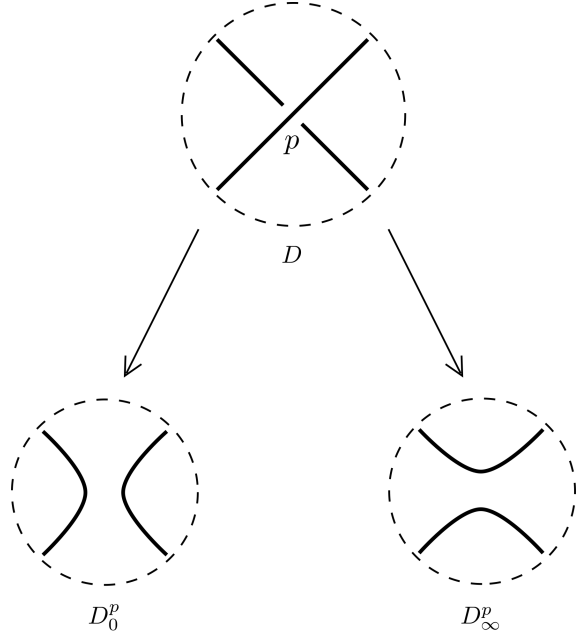}
\caption{link diagrams $D_0^p$ and $D_\infty^p$}
\label{Kauffman}
\end{figure}
We assume that the crossing data $D^{-1}(m)$ on $m\in M$ has the order 
\[
 \cdots<a_1=a_2=\dots=a_i<b=b_1=b_2=\dots=b_j<c_1=c_2=\dots=c_k<\cdots.
\]
Then we define link diagrams $D_{+}^b$ and $D_{-}^b$ by changing the order of $D^{-1}(m)$.
The link diagrams $D_{+}^b$ and $D_{-}^b$ have the same underlying immersion with one of $D$.
The crossing data of $(D_{+}^b)^{-1}(m)$ is
\[
 \cdots<a_1=a_2=\dots=a_i<b_1=b_2=\dots=b_j<b<c_1=c_2=\dots=c_k<\cdots
\]
and the crossing data of $(D_{-}^b)^{-1}(m)$
\[
 \cdots<a_1=a_2=\dots=a_i<b<b_1=b_2=\dots=b_j<c_1=c_2=\dots=c_k<\cdots.
\]
\begin{DEF}\label{skeinalgebra}
\emph{The skein algebra} $\SKQ{\Sigma}$ on $\Sigma$ is the quotient of $\mathbf{K}[q^{\frac{1}{2}}, q^{-\frac{1}{2}}]\mathscr{L}(\Sigma)$ by the ideal generated by the following relations:
\begin{enumerate}
\item $\left[D\right]=q\left[D_0^p\right]+q^{-1}\left[D_{\infty}^p\right]$ for any link diagram $D$ and its intersection point $p$ in $\operatorname{int}S$ 
\quad(\emph{the Kauffman bracket skein relation}),
\item $q^{\frac{1}{2}(\#\ell(b)-\#r(b))}\left[D_{-}^b\right]=\left[D\right]=q^{\frac{1}{2}(\#r(b)-\#\ell(b))}\left[D_{+}^b\right]$ for any link diagram $D$ and its intersection point $m$ in $M$ 
\quad(\emph{the boundary skein relation}),
\item $\left[\mathcal{O}\right]=-(q^2+q^{-2})\left[\emptyset\right]$ 
where $\mathcal{O}$ is the curl on $\operatorname{int}S$,
\item $\left[\mathcal{O}_m\right]=0$ for any $m$ in $M$ 
where $\mathcal{O}_m$ is the monogon with vertex $m$.
\end{enumerate}
\end{DEF}
We remark that the ideal generated by the relation $A=B$ means that the ideal generated by $A-B$.

We prepare a Poisson algebra whose quantization is $\SKQ{\Sigma}$.
The Poisson algebra denoted by $\SK{\Sigma}$ is obtained from $\UGPA{\Sigma}^{+}|_{T=-1}$.
\begin{DEF}
The Poisson $\mathbf{K}$-algebra $\SK{\Sigma}$ is the quotient of the algebra $\UGSA{\Sigma}^{+}|_{T=-1}$ by the ideal generated by the following relations:
\begin{enumerate}
\item $\left|D\right|=\left|D_0^p\right|+\left|D_\infty^p\right|$ for any underlying immersion of a link diagram $D$ and its intersection point $p\in\operatorname{int}S$,
\item $\left|\mathcal{O}\right|=-2$.
\end{enumerate}
We remark that the relation (1) is independent of the crossing data of $D$.
We can easily confirm that the ideal is the Poisson ideal of $\UGPA{\Sigma}^{+}|_{T=-1}$.
Therefore, we can obtain a Poisson structure on $\SK{\Sigma}$.
\end{DEF}
We calculate the Poisson bracket on $\SK{\Sigma}$ according to the definition of Poisson bracket (\ref{unoriPoisson}).
Let us denote by $\left|D\right|$ the element of $\SK{\Sigma}$ represented by generic curves $D$.
Let $\{a,b\}\subset\overline{\mathsf{Curves}}(\Sigma)$ be generic curves.
We assume that the intersection points of $a$ and $b$ in $\operatorname{int}S$ is $\{p_1, p_2,\dots, p_n\}$.
Then,
\begin{align*}
\left|\mathsf{E}_{p_i}(a,b)\right|
=&\sum_{\mathsf{X}\in\{\mathsf{E},\mathsf{e}\}}
\left|\mathsf{X}_1\mathsf{X}_2\cdots\mathsf{X}_{i-1}\mathsf{E}_i\mathsf{X}_{i+1}\mathsf{X}_{i+2}\cdots\mathsf{X}_n(a,b)\right|,\\
\left|\mathsf{e}_{p_i}(a,b)\right|
=&\sum_{\mathsf{X}\in\{\mathsf{E},\mathsf{e}\}}
\left|\mathsf{X}_1\mathsf{X}_2\cdots\mathsf{X}_{i-1}\mathsf{e}_i\mathsf{X}_{i+1}\mathsf{X}_{i+2}\cdots\mathsf{X}_n(a,b)\right|,\\
\left|a\right|\left|b\right|
=&\sum_{\mathsf{X}\in\{\mathsf{E},\mathsf{e}\}}
\left|\mathsf{X}_1\mathsf{X}_2\cdots\mathsf{X}_n(a,b)\right|.
\end{align*}
Therefore,
\begin{align*}
\{\,\left|a\right|,\left|b\right|\,\}
=&\sum_{i=1}^n\sum_{\mathsf{X}\in\{\mathsf{E},\mathsf{e}\}}
\left|\mathsf{X}_1\cdots\mathsf{X}_{i-1}\mathsf{E}_i\mathsf{X}_{i+1}\cdots\mathsf{X}_n(a,b)\right|
-\left|\mathsf{X}_1\cdots\mathsf{X}_{i-1}\mathsf{e}_i\mathsf{X}_{i+1}\cdots\mathsf{X}_n(a,b)\right|\\
&+\sum_{\mathsf{X}\in\{\mathsf{E},\mathsf{e}\}}
\frac{1}{2}(n_{+}(a,b)-n_{-}(a,b))\left|\mathsf{X}_1\mathsf{X}_2\cdots\mathsf{X}_n(a,b)\right|.
\end{align*}
We observe a correspondence between the set of symbols $X=\{\,\mathsf{X}_1\mathsf{X}_2\cdots\mathsf{X}_n\mid \mathsf{X}=\mathsf{E},\mathsf{e}\,\}$ 
and $E=\bigsqcup_{i=1}^n\{\,\mathsf{X}_1\mathsf{X}_2\cdots\mathsf{X}_{i-1}\mathsf{E}_i\mathsf{X}_{i+1}\cdots\mathsf{X}_n\mid \mathsf{X}=\mathsf{E},\mathsf{e}\,\}$.
Fix an element $\mathsf{Y}_1\mathsf{Y}_2\cdots\mathsf{Y}_n\in X$, 
then the number of elements in $E$ which coincide with $\mathsf{Y}_1\mathsf{Y}_2\cdots\mathsf{Y}_n$ is the number of symbols $\mathsf{E}_i$ appearing in $\mathsf{Y}_1\mathsf{Y}_2\cdots\mathsf{Y}_n$.
Consequently, we can describe the bracket of $\SK{\Sigma}$ as the following,
\begin{align}\label{bracket-SKH}
\begin{split}
\{\,\left|a\right|,\left|b\right|\,\}
=&\sum_{\mathsf{X}\in\{\mathsf{E},\mathsf{e}\}}
(\#\{\mathsf{X}_i=\mathsf{E}_i\}-\#\{\mathsf{X}_i=\mathsf{e}_i\})
\left|\mathsf{X}_1\mathsf{X}_2\cdots\mathsf{X}_n(a,b)\right|\\
&\quad+\sum_{\mathsf{X}\in\{\mathsf{E},\mathsf{e}\}}
\frac{1}{2}(n_{+}(a,b)-n_{-}(a,b))\left|\mathsf{X}_1\mathsf{X}_2\cdots\mathsf{X}_n(a,b)\right|
\end{split}
\end{align}
where $\#\{\mathsf{X}_i=\mathsf{E}_i\}$ (resp. $\#\{\mathsf{X}_i=\mathsf{e}_i\}$) is the number of $\{\mathsf{E}_i\mid i\in\{1,2,\dots,n\}\}$ (resp. $\{\mathsf{e}_i\mid i\in\{1,2,\dots,n\}\}$) appearing in $\mathsf{X}_1\mathsf{X}_2\cdots\mathsf{X}_n$.

Let $\mathbb{C}[[h]]$ be the ring of complex formal power series in one variable $h$.
We consider the skein algebra of $\Sigma$ with coefficient in $\mathbb{C}[[h]]$, 
that is, the quotient of $\mathbb{C}[[h]]\mathscr{L}(\Sigma)^{+}$ such that its skein relations are obtained by substituting $q=\exp(h/2)$.
We denote by $\SKH{\Sigma}$ the skein algebra with coefficient in $\mathbb{C}[[h]]$.
We also use the same symbol as a framed link $[D]$ to represent the element in $\SKH{\Sigma}$.

\emph{A simple multicurve} $D$ is an immersion of generic curves on $\Sigma$ with no intersection points in $\operatorname{int}S$, no curls and no monogons with vertices in $M$.
We define the framed link of the simple multicurve $D$ as a framed link $\left[D\right]$ with only simultaneous-crossings.
We denote by $\mathsf{SMulti}(\Sigma)$ the set of links of simple multicurves on $\Sigma$.
\begin{LEM}[Muller~\cite{Muller} Lemma 4.1.]\label{freeness}
The skein algebra $\SKH{\Sigma}$ is a topologically free module.
\end{LEM}
\begin{proof}
We construct a $\mathbb{C}[[h]]$-linear map from $\SKH{\Sigma}$ to the set of formal power series with coefficient in $\mathbb{C}\mathsf{SMulti}(\Sigma)$ where $\mathbb{C}\mathsf{SMulti}(\Sigma)$ is the complex vector space spanned by $\mathsf{SMulti}(\Sigma)$.
For any link diagram $D$, 
we expand $[D]$ in a series of links of simple multicurves in $\SKH{\Sigma}$.
First, 
we can eliminate intersection points in $\operatorname{int}S$ of $D$ by using the Kauffman bracket skein relation.
Next, 
the crossings of $D$ on $M$ can be deformed into simultaneous-crossings by using the boundary skein relation.
Finally, 
we vanish curls and monogons by relation (3) and (4) in Definition~\ref{skeinalgebra}.
Consequently, 
$[D]$ can be expanded in a formal power series with coefficient in $\mathbb{C}\mathsf{SMulti}(\Sigma)$.
The local moves of link diagrams is realized by the skein relation and the boundary skein relation;
\begin{align*}
\vcenter{\hbox{\includegraphics[width=1cm]{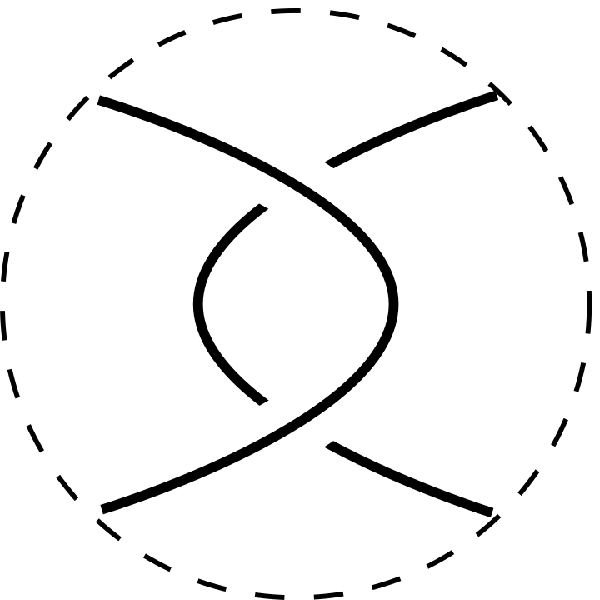}}}\ 
=&q\ \vcenter{\hbox{\includegraphics[width=1cm]{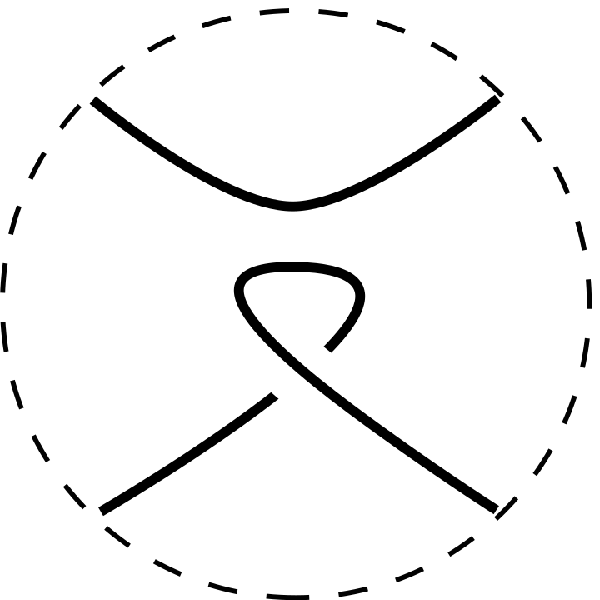}}}\ 
+q^{-1}\ \vcenter{\hbox{\includegraphics[width=1cm]{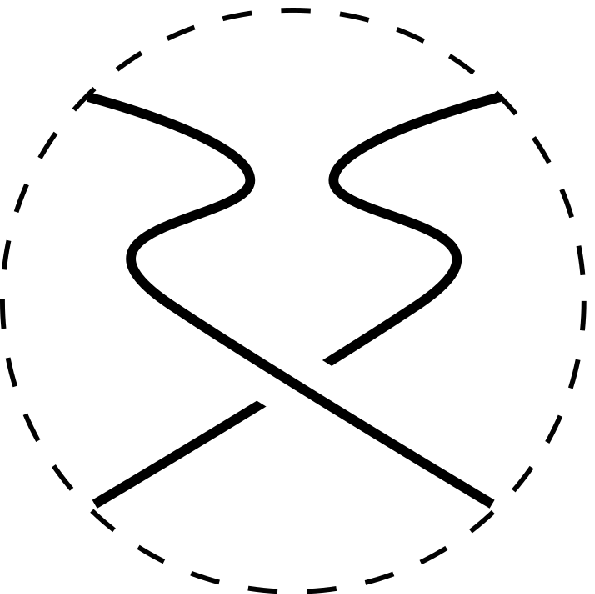}}}\ \\
=&q^2\ \vcenter{\hbox{\includegraphics[width=1cm]{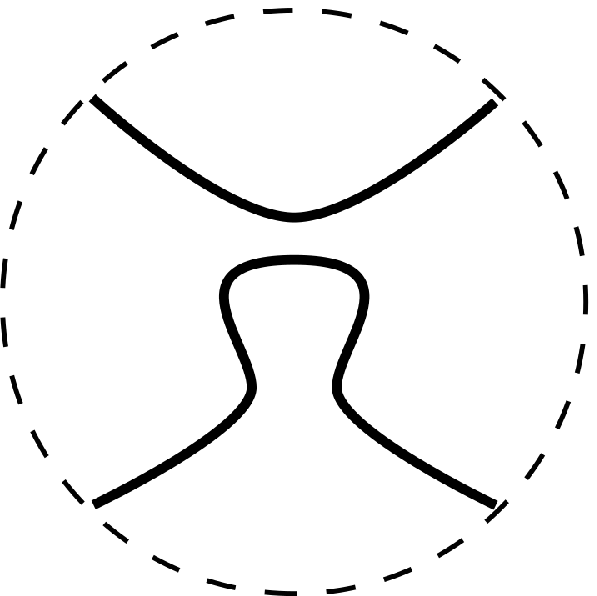}}}\ 
+\ \vcenter{\hbox{\includegraphics[width=1cm]{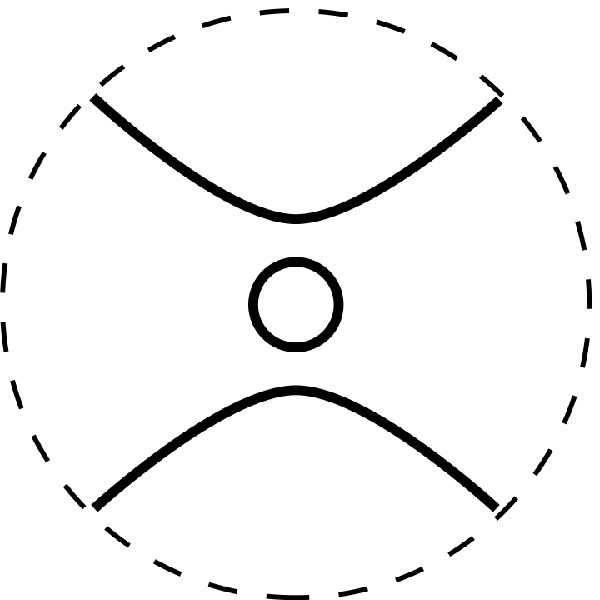}}}\ 
+\ \vcenter{\hbox{\includegraphics[width=1cm]{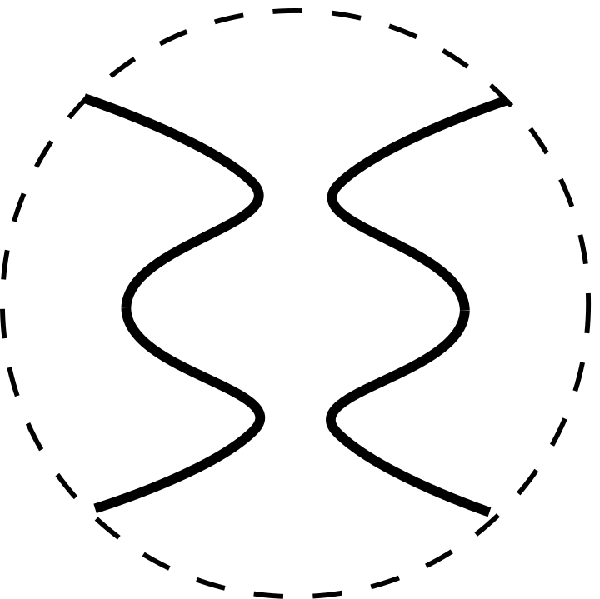}}}\ 
+q^{-2}\ \vcenter{\hbox{\includegraphics[width=1cm]{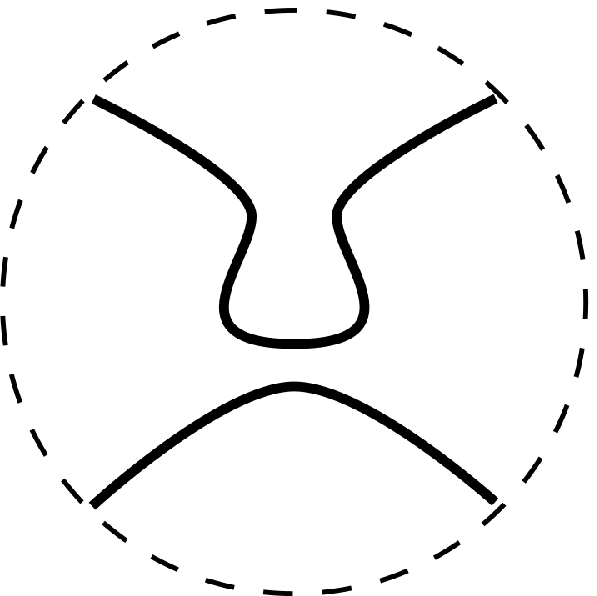}}}\ \\
=&q^2\ \vcenter{\hbox{\includegraphics[width=1cm]{R23}}}\ 
-(q^2+q^{-2})\ \vcenter{\hbox{\includegraphics[width=1cm]{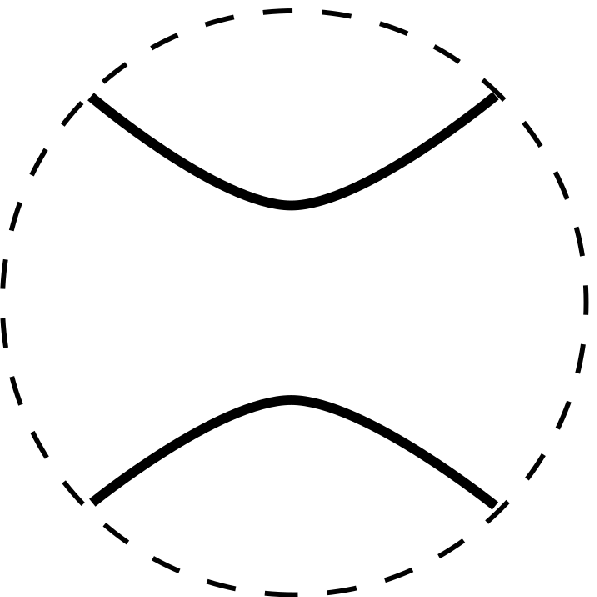}}}\ 
+\ \vcenter{\hbox{\includegraphics[width=1cm]{R26}}}\ 
+q^{-2}\ \vcenter{\hbox{\includegraphics[width=1cm]{R27}}}\ \\
=&\ \vcenter{\hbox{\includegraphics[width=1cm]{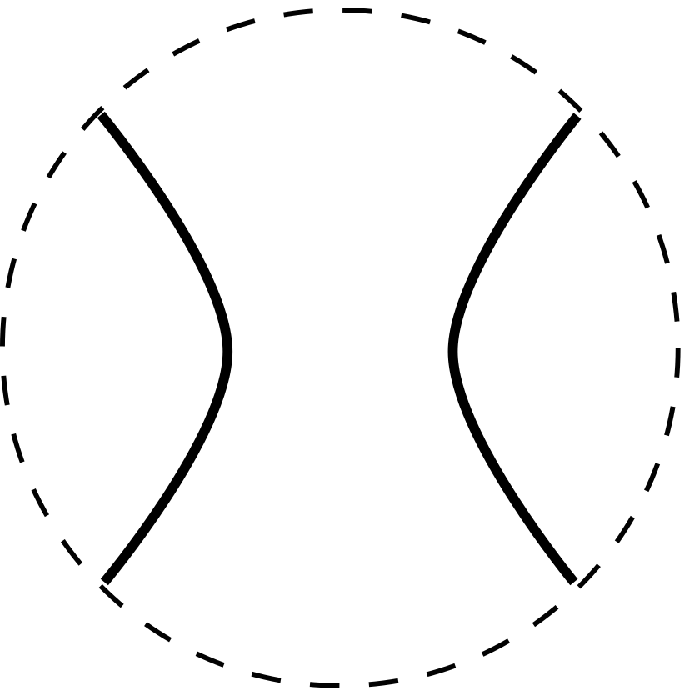}}}\ ,\\
\vcenter{\hbox{\includegraphics[width=1cm]{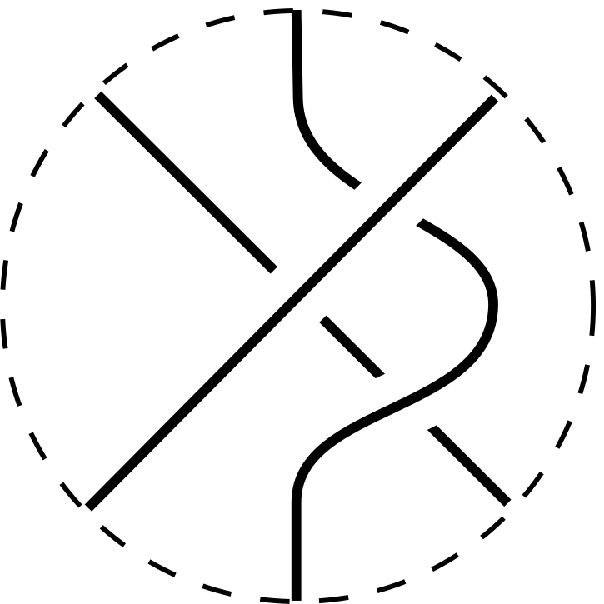}}}\ 
=&q\ \vcenter{\hbox{\includegraphics[width=1cm]{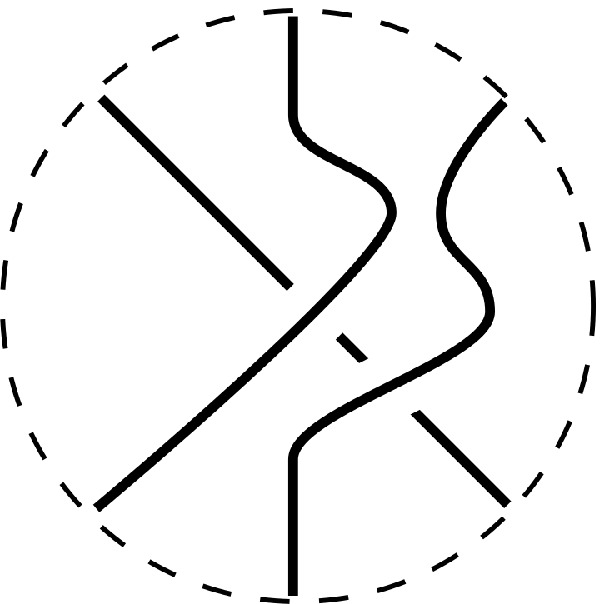}}}\ 
+q^{-1}\ \vcenter{\hbox{\includegraphics[width=1cm]{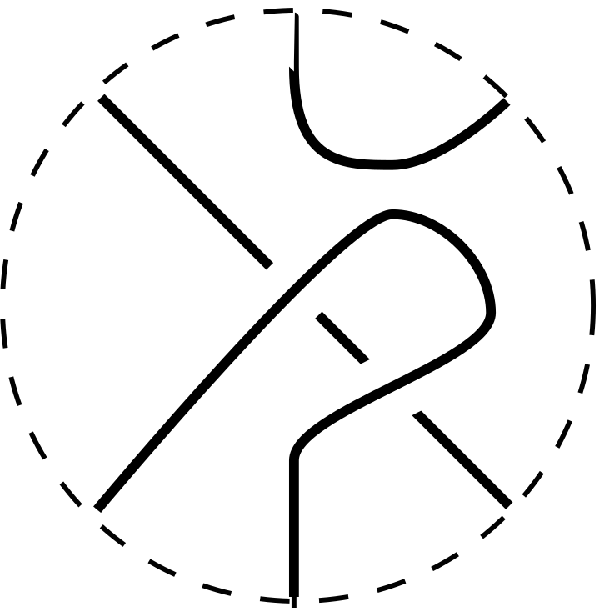}}}\ \\
=&q\ \vcenter{\hbox{\includegraphics[width=1cm]{R31r}}}\ 
+q^{-1}\ \vcenter{\hbox{\includegraphics[width=1cm]{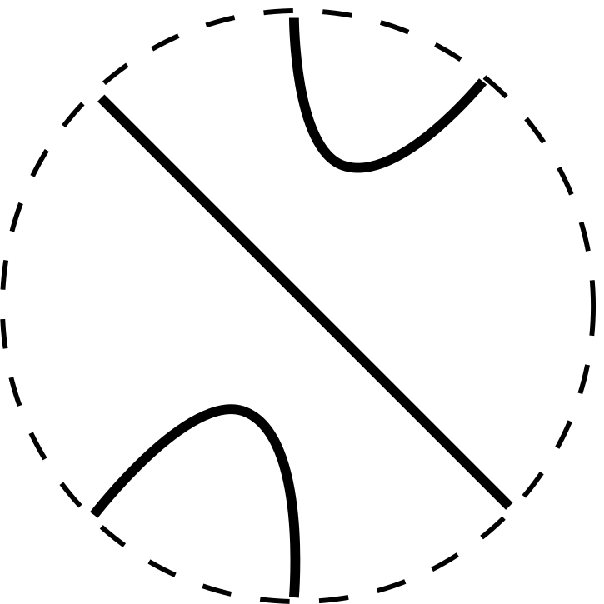}}}\ \\
=&q\ \vcenter{\hbox{\includegraphics[width=1cm]{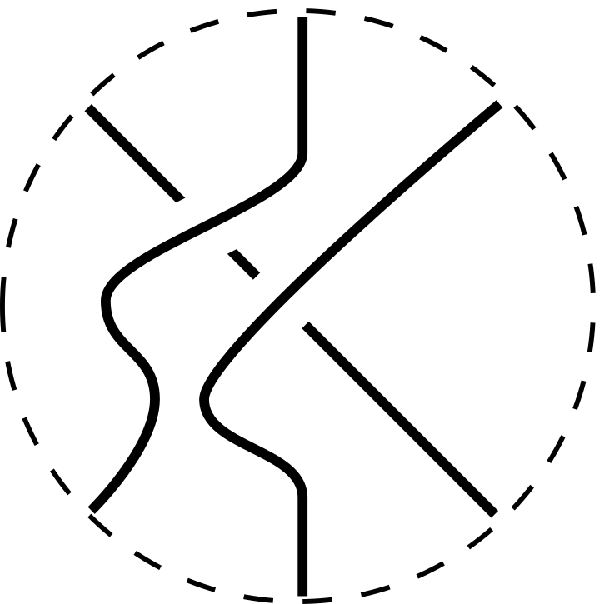}}}\ 
+q^{-1}\ \vcenter{\hbox{\includegraphics[width=1cm]{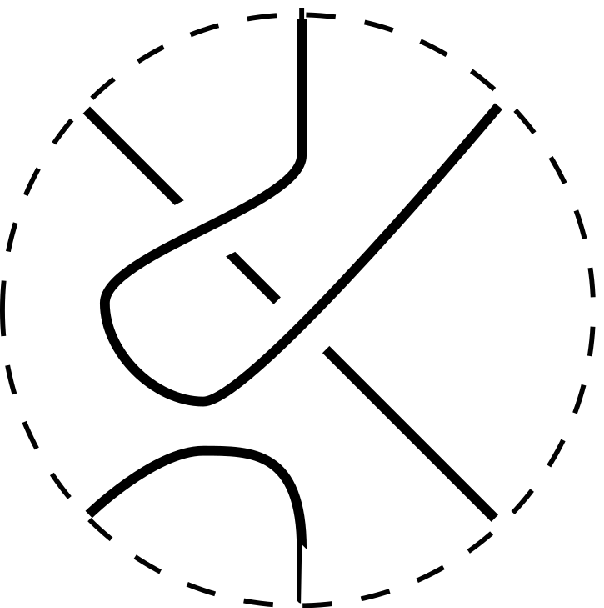}}}\ \\
=&\ \vcenter{\hbox{\includegraphics[width=1cm]{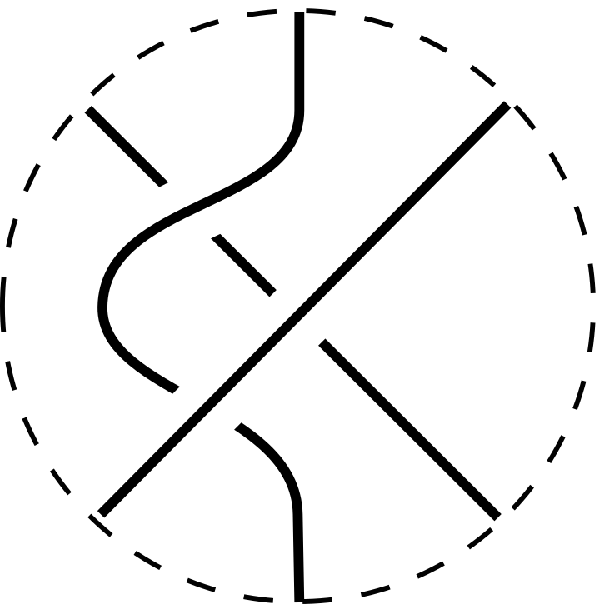}}}\ ,\\
\vcenter{\hbox{\includegraphics[width=1cm]{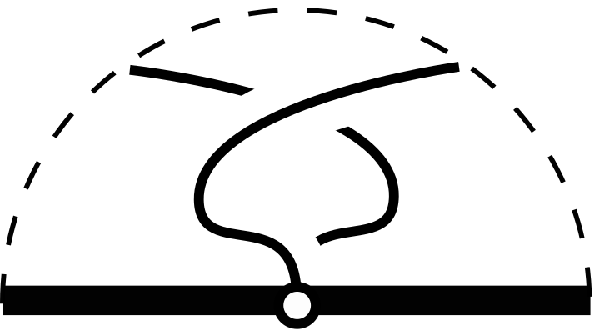}}}\ 
=&q\ \vcenter{\hbox{\includegraphics[width=1cm]{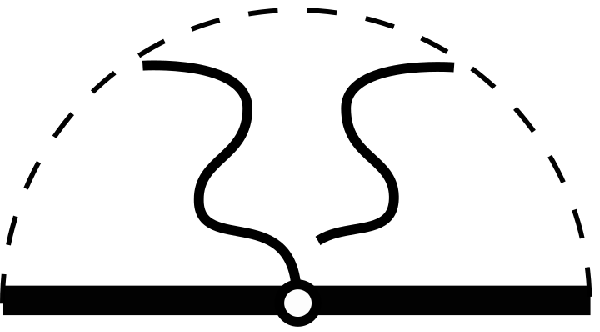}}}\ 
+q^{-1}\ \vcenter{\hbox{\includegraphics[width=1cm]{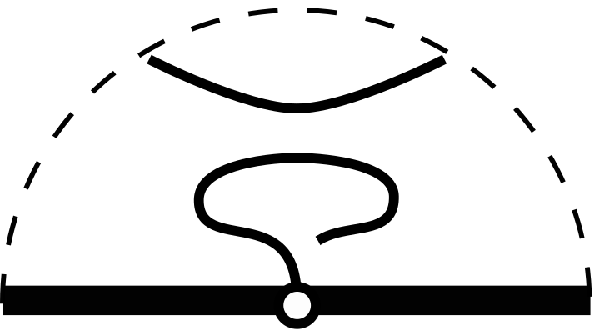}}}\ \\
=&q\ \vcenter{\hbox{\includegraphics[width=1cm]{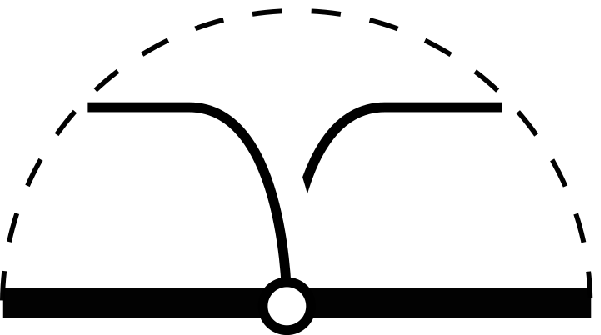}}}\ \\
=&\ \vcenter{\hbox{\includegraphics[width=1cm]{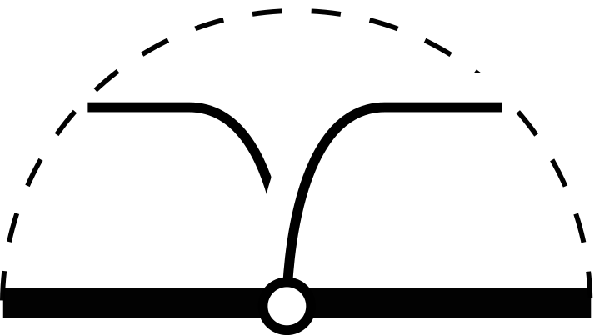}}}\ .
\end{align*}
Therefore, we get the same expansion of $[D]$ for another link diagram $D'$ which is a representative of the link $[D]$.

We can describe any element $f\in\SKH{\Sigma}$ as
\[
 f=\sum_{n\geq 0}
\left(
\sum_{D\in\mathscr{D}(\Sigma)}a_{D}^{(n)}[D]
\right)h^{n}
\]
where $a_{D}^{(n)}\in\mathbb{C}$.
By expanding $[D]$ as the above, 
the element in $\mathbb{C}\mathsf{SMulti}(\Sigma)[[h]]$ is uniquely determined.
We denote by $\Phi\colon\SKH{\Sigma}\to\mathbb{C}\mathsf{SMulti}(\Sigma)[[h]]$ the $\mathbb{C}[[h]]$-linear extension of the above map where $\mathbb{C}\mathsf{SMulti}(\Sigma)[[h]]$ is the formal power series with coefficient in $\mathbb{C}\mathsf{SMulti}(\Sigma)$.
The inverse map $\Psi\colon\mathbb{C}\mathsf{SMulti}(\Sigma)[[h]]\to\SKH{\Sigma}$ is induced by the composed map of 
the inclusion $\mathsf{SMulti}(\Sigma)$ into $\mathscr{L}(\Sigma)^{+}$ and the projection of $\mathscr{L}(\Sigma)^{+}$ to $\SKH{\Sigma}$.
It is clear that $\Phi\circ\Psi=\text{id}_{\SKH{\Sigma}}$ and $\Psi\circ\Phi=\text{id}_{\mathbb{C}\mathsf{SMulti}(\Sigma)[[h]]}$.
\end{proof}

\begin{THM}\label{quantization}
The skein algebra $\SKH{\Sigma}$ gives a quantization of the Poisson algebra $\SK{\Sigma}$.
\end{THM}
\begin{proof}
We can construct the $\mathbb{C}$-linear isomorphism $\SK{\Sigma}\to\mathbb{C}\mathsf{SMulti}(\Sigma)$ in the same way as the construction of the map $\Phi$ in the proof of Lemma~\ref{freeness}.
Therefore, 
$\SKH{\Sigma}$ is isomorphic to $\SK{\Sigma}[[h]]$ through $\mathbb{C}\mathsf{SMulti}(\Sigma)[[h]]$.

Comparing the relations of $\SKH{\Sigma}$ and $\SK{\Sigma}$, 
the algebra homomorphism $\tilde{p}\colon\SKH{\Sigma}\to\SK{\Sigma}$ is obtained by $h\mapsto 0$ and $[D]\mapsto\left|D\right|$. 
The algebra isomorphism $p\colon\SKH{\Sigma}/h\SKH{\Sigma}\to\SK{\Sigma}$ is induced by $\tilde{p}$ because $\ker\tilde{p}\cong\ker \tilde{p}\circ\Psi=h\mathbb{C}\mathsf{SMulti}(\Sigma)[[h]]\cong h\SKH{\Sigma}$.
We need to confirm the equation $\{\tilde{p}(x), \tilde{p}(y)\}=\tilde{p}((xy-yx)/h)$ for any $x,y\in\SKH{\Sigma}$.
For any link diagram $D\in\mathscr{D}(\Sigma)$, we can describe $[D]\in\SKH{\Sigma}$ as a sum of multiplications of loops and arcs by using the skein relation and the boundary skein relation.
Therefore, we only have to calculate the commutator $\tilde{p}(([a][b]-[b][a])/h)$ such that $a$ and $b$ are generic curves in $\Sigma$.
We define a link diagram $D(a,b)$ as a link diagram replacing over and under crossings of $a\cdot b$ with simultaneous-crossings at marked points. 
We denote the strands contained in the neiborhood of M of $D(a,b)$ by $a_1, a_2, b_1$ and $b_2$ where $a_i$ (resp. $b_i$) is contained in $a$ (resp. $b$) for $i=1,2$. 
Then we obtaine the following equations by using boundary skein relation: 
\begin{align*}
[a][b]=[a\cdot b]
=&q^{\frac{1}{2}(\#\ell(a_1)-\#r(a_1))}q^{\frac{1}{2}(\#\ell(a_2)-\#r(a_2))}
[D(a,b)],\\
[b][a]=[b\cdot a]
=&q^{\frac{1}{2}(\#r(a_1)-\#\ell(a_1))}q^{\frac{1}{2}(\#r(a_2)-\#\ell(a_2))}
[D(b,a)].
\end{align*}
We denote the intersection points of $a$ and $b$ in $\operatorname{int}S$ by 
$\{p_1,p_2,\dots,p_n\}$.
Then we eliminate the crossings at $p_i$ of $D(a,b)$ and $D(b,a)$ by using the skein relation;
\begin{align*}
[D(a,b)]
=&\sum_{\varepsilon_i\in\{0,\infty\}}q^{\#\{\varepsilon_i=0\}-\#\{\varepsilon_i=\infty\}}
[(\cdots((D(a,b)_{\varepsilon_1}^{p_1})_{\varepsilon_2}^{p_2})\cdots)_{\varepsilon_n}^{p_n}],\\
[D(b,a)]
=&\sum_{\varepsilon_i'\in\{0,\infty\}}q^{\#\{\varepsilon_i'=0\}-\#\{\varepsilon_i'=\infty\}}
[(\cdots((D(b,a)_{\varepsilon_1'}^{p_1})_{\varepsilon_2'}^{p_2})\cdots)_{\varepsilon_n'}^{p_n}].
\end{align*}
The elimination $\varepsilon_i=0$ (resp. $\varepsilon_i=\infty$) of $p_i$ of $D(a,b)$ coincides with the elimination $\varepsilon_i'=\infty$ (resp. $\varepsilon_i'=0$) of $p_i$ of $D(b,a)$ at the neighborhood of $p_i$.
Therefore,
\begin{align}\label{noncomm}
[a][b]-[b][a]
=\sum_{\varepsilon_i\in\{0,\infty\}}(q^{\varepsilon+\frac{1}{2}N}-q^{-\varepsilon-\frac{1}{2}N})[(\cdots((D(a,b)_{\varepsilon_1}^{p_1})_{\varepsilon_2}^{p_2})\cdots)_{\varepsilon_n}^{p_n}]
\end{align}
where $\varepsilon=\#\{\varepsilon_i=0\}-\#\{\varepsilon_i=\infty\}$ 
and $N=\sum_{i=1}^2\#\{\ell(a_i)\}-\#\{r(a_i)\}$.
We calculate (\ref{noncomm}) by substituting $q=\exp(h/2)$ modulo $h^2\SKH{\Sigma}$.
Then, 
\[
 [a][b]-[b][a]=\sum_{\varepsilon_i\in\{0,\infty\}}\left(\varepsilon+\frac{N}{2}\right)h[(\cdots((D(a,b)_{\varepsilon_1}^{p_1})_{\varepsilon_2}^{p_2})\cdots)_{\varepsilon_n}^{p_n}] \pmod{h^2\SKH{\Sigma}}.
\]
Considering underlying immersion of $[(\cdots((D(a,b)_{\varepsilon_1'}^{p_1})_{\varepsilon_2'}^{p_2})\cdots)_{\varepsilon_n'}^{p_n}]$, we get
\[
 \tilde{p}([(\cdots((D(a,b)_{\varepsilon_1'}^{p_1})_{\varepsilon_2'}^{p_2})\cdots)_{\varepsilon_n'}^{p_n}])=\left|\mathsf{X}_{p_n}\cdots\mathsf{X}_{p_2}\mathsf{X}_{p_1}(a,b)\right|
\]
where $\mathsf{X}_{p_i}=\mathsf{E}_{p_i}$ if $\varepsilon_i=0$ and $\mathsf{X}_{p_i}=\mathsf{e}_{p_i}$ if $\varepsilon_i=\infty$.
We can also show that
\[
 \varepsilon=\#\{\mathsf{X}_{p_i}=\mathsf{E}_{p_i}\}-\#\{\mathsf{X}_{p_i}=\mathsf{e}_{p_i}\}
\]
and
\[
 N=n_{+}(a,b)-n_{-}(a,b).
\]
Therefore,
\begin{align*}
\tilde{p}\left(\frac{[a][b]-[b][a]}{h}\right)
=&\sum_{\mathsf{X}\in\{\mathsf{E},\mathsf{e}\}}
(\#\{\mathsf{X}_i=\mathsf{E}_i\}-\#\{\mathsf{X}_i=\mathsf{e}_i\})
\left|\mathsf{X}_1\mathsf{X}_2\cdots\mathsf{X}_n(a,b)\right|\\
&\quad+\sum_{\mathsf{X}\in\{\mathsf{E},\mathsf{e}\}}
\frac{1}{2}(n_{+}(a,b)-n_{-}(a,b))\left|\mathsf{X}_1\mathsf{X}_2\cdots\mathsf{X}_n(a,b)\right|\\
=&\{\,\left|a\right|,\left|b\right|\,\}\quad\text{by (\ref{bracket-SKH})}.
\end{align*}
Thus, 
the map $\tilde{p}$ induces an isomorphism of Poisson algebras $\SKH{\Sigma}/h\SKH{\Sigma}$ 
and $\SK{\Sigma}$.
\end{proof}

\section{Appendix}\label{sec;Appendix}
In this section, we review the definition of a (bi-, co-)Poisson (bi-, co-)algebra.

Let $R$ be a commutative ring with unit and  $S$ a module over $\mathbf{R}$. 
We assume that $S$ is a commutative associative algebra equipped with a multiplication $m\colon S\otimes S\to S$ and a linear map $\nabla=\{\cdot,\cdot\}\colon S\otimes S\to S$. 
The symbol $\otimes$ denotes the tensor product over $\mathbf{R}$ in this section.
\begin{DEF}
The pair $(S,\nabla)$ is a Poisson algebra if it satisfies the following conditions.
\begin{enumerate}
\item $\nabla\circ\sigma_{1,2}=-\nabla$ \quad (\emph{the skew symmetry}), 
\item $\nabla\circ (\nabla\otimes\textup{id})\circ(\tau^2+\tau+\textup{id}_{S^{\otimes 3}})=0$ \quad (\emph{the Jacobi identity}), 
\item $\nabla\circ(\textup{id}\otimes m)=m\circ(\nabla\otimes\textup{id}+(\textup{id}\otimes\nabla)\circ\sigma_{1,2})$ \quad (\emph{the Leibniz rule})
\end{enumerate}
where $\sigma_{i,j}$ is an linear operator which exchanges $i$-th and $j$-th tensor components and $\tau=\sigma_{2,3}\circ\sigma_{1,2}$. We call the linear map $\nabla$ \emph{the Poisson bracket}. 
We can describe above conditions explicitly as the following. For any $a,b,c$ in $S$
\begin{enumerate}
\item $\{a,b\}=-\{b,a\}$, 
\item $\{a,\{b,c\}\}+\{b,\{c,a\}\}+\{c,\{a,b\}\}=0$, 
\item $\{a,bc\}=\{a,b\}c+b\{a,c\}$.
\end{enumerate}
\end{DEF}
\begin{EX}
The symmetric algebra $\operatorname{Sym}(\mathfrak{g})$ of a Lie algebra $\mathfrak{g}$ over $\mathbf{R}$ is a Poisson algebra. By using the Leibniz rule, the Lie bracket of $\mathfrak{g}$ is extended to the Poisson bracket on $\operatorname{Sym}(\mathfrak{g})$.
\end{EX}

We assume that $S$ is a cocommutative coassosiative coalgebra equipped with a comultiplication $\Delta\colon S\to S\otimes S$ and a linear map $\delta\colon S\to S\otimes S$.
\begin{DEF}
The pair $(S,\delta)$ is a co-Poisson coalgebra if it satisfies the following conditions.
\begin{enumerate}
\setcounter{enumi}{3}
\item $\sigma_{1,2}\circ\delta=-\delta$ \quad (\emph{the co-skew symmetry}), 
\item $(\tau^2+\tau+\textup{id}_{S^{\otimes 3}})\circ(\textup{id}\otimes\delta)\circ\delta=0$ \quad (\emph{the co-Jacobi identity}),
\item $(\textup{id}\otimes\Delta)\circ\delta=(\delta\otimes\textup{id}+\sigma_{1,2}\circ(\textup{id}\otimes\delta))\circ\Delta$ \quad (\emph{the co-Leibniz rule}).
\end{enumerate}
\end{DEF}
We call the linear map $\delta$ \emph{the Poisson cobracket}. The definition of a co-Poisson coalgebra is the dual of a Poisson algebra.
\begin{DEF}
A bialgebra $(S,m,\Delta)$ is a bi-Poisson bialgebra if $S$ equips the above linear maps $\nabla$ and $\delta$ which satisfy condition (1)--(6) and the following compatibility conditions.
\begin{enumerate}
\setcounter{enumi}{6}
\item $\Delta\circ\nabla=\nabla_{S\otimes S}\circ\sigma_{2,3}\circ\Delta_{S\otimes S}$ \quad (\emph{the $(\Delta,\nabla)$-compatibility}),
\item $\delta\circ m=m_{S\otimes S}\circ\sigma_{2,3}\circ\delta_{S\otimes S}$ \quad (\emph{the $(m,\delta)$-compatibility}),
\item $\delta\circ\nabla=\nabla_{S\otimes S}\circ\sigma_{2,3}\circ\delta_{S\otimes S}$ \quad (\emph{the $(\nabla,\delta)$-compatibility})
\end{enumerate}
where $m_{S\otimes S}=(m\otimes m)\circ\sigma_{2,3}$, $\Delta_{S\otimes S}=\sigma_{2,3}\circ(\Delta\otimes\Delta)$, $\nabla_{S\otimes S}=(\nabla\otimes m+m\otimes\nabla)\circ\sigma_{2,3}$ and $\delta_{S\otimes S}=\sigma_{2,3}\circ(\delta\otimes\Delta+\Delta\otimes\delta)$.
\end{DEF}
\begin{EX}
The symmetric algebra $\operatorname{Sym}(\mathfrak{g})$ of a Lie coalgebra $(\mathfrak{g},\delta)$ over $\mathbf{R}$ is a co-Poisson bialgebra. The comultiplication of $\operatorname{Sym}(\mathfrak{g})$ is defined by extending $\Delta(x)=x\otimes 1+1\otimes x$ for any $x$ in $\mathfrak{g}$ as an algebra homomorphism. By using the $(m,\delta)$-compatibility, the Lie cobracket of $\mathfrak{g}$ is extended to the Poisson cobracket on $\operatorname{Sym}(\mathfrak{g})$.
\end{EX}
\bibliographystyle{amsplain}
\bibliography{yuasa1}
\end{document}